\documentclass[12pt,leqno,a4paper]{article}

\usepackage[english]{babel}
\usepackage{amsmath}
\usepackage{amsfonts}
\usepackage{amsthm}
\usepackage{amssymb}

\usepackage{xspace}
\usepackage{euscript}
\usepackage{graphicx}
\usepackage{amscd}
\usepackage{tabularx}

\usepackage{enumerate}
 \usepackage{epsfig} 
 \usepackage{graphics} 
\theoremstyle{plain}
\newtheorem{theo}{Theorem}[section]
\newtheorem{lem}[theo]{Lemma}
\newtheorem{prop}[theo]{Proposition}
\theoremstyle{definition}
\newtheorem{definition}[theo]{Definition}
\theoremstyle{remark}
\newtheorem{rem}[theo]{Remark}
\numberwithin{equation}{section}

\newcommand{\C}{\mathbb{C}}

\newcommand{\R}{\mathbb{R}}
\newcommand{\N}{\mathbb{N}}
\newcommand{\M}{\mathbb{M}}
\newcommand{\divrg}{\textrm{div}\,}

\title{Stable determination of a rigid inclusion in an anisotropic elastic plate
\thanks{Work supported by PRIN No. 20089PWTPS}}
\author{Antonino Morassi\thanks{Dipartimento di Ingegneria Civile e Architettura,
Universit\`a degli Studi di Udine, via Cotonificio 114, 33100
Udine, Italy. E-mail: \textsf{antonino.morassi@uniud.it}}, \  Edi
Rosset\thanks{Dipartimento di Matematica e Informatica,
Universit\`a degli Studi di Trieste, via Valerio 12/1, 34127
Trieste, Italy. E-mail: \textsf{rossedi@univ.trieste.it}} \ and
Sergio Vessella\thanks{DIMAD, Universit\`a degli Studi di Firenze,
Via Lombroso 6/17, 50134 Firenze, Italy,
\textsf{sergio.vessella@dmd.unifi.it}}}

\begin{document}

\maketitle

\begin{abstract}

In this paper we consider the inverse problem of determining a
rigid inclusion inside a thin plate by applying a couple field at
the boundary and by measuring the induced transversal displacement
and its normal derivative at the boundary of the plate. The plate
is made by non-homogeneous linearly elastic material belonging to
a general class of anisotropy. For this severely ill-posed
problem, under suitable a priori regularity assumptions on the
boundary of the inclusion, we prove a stability estimate of
log-log type.

\medskip

\noindent\textbf{Mathematics Subject Classification (2010)}:
Primary 35B60. Secondary 35B30, 35Q74, 35R30.

\medskip

\noindent \textbf{Keywords}: Inverse problems, elastic plates,
stability estimates, unique continuation, rigid inclusion.
\end{abstract}

\centerline{}

\section{Introduction}

The present study continues a line of research on Non-Destructive
Techniques (NDT) in the theory of elastic plates. Specifically,
the paper deals with the determination of a rigid inclusion
embedded in elastic plates by measurements taken at the boundary.
This inverse problem arises in damage assessment of plates which
are possibly defective due to the presence of interior rigid
inclusions induced during the manufacturing process. The most
simple and common NDT is the visual inspection of the specimen.
However, the visual inspection is unable to detect a rigid
inclusion which is embedded in a plate or which is
indistinguishable {}from the surrounding material. Other
conventional NDTs, such as those based on thermal or ultrasonic
analysis, X-rays methods and others, are local by nature. To be
effective, these methods require the vicinity of the inclusion be
known a priori and readily accessible for testing, whereas in many
practical situations the boundary of the plate is the only
accessible region. Vibration response-based identification methods
were also recently developed to overcome these difficulties, see,
for example, \cite{F-Q-11}.

The basic idea of most NDTs is that the inclusion changes the
response of a plate and that the defect can be determined by
comparing the response of a possibly defective specimen with a
reference specimen, i.e. a plate without inclusion. In the case of
inclusion made by (different) elastic material, in \cite{M-R-V3},
\cite{M-R-V5} constructive upper and lower estimates of the area
of the inclusion in terms of the difference between the work
exerted in deforming the defective and the reference plate by
applying a given couple field at the boundary were determined.
These results give no indication on the position of the domain
occupied by the inclusion and, unfortunately, they still have not
been extended to the case of a rigid inclusion. In this paper we
tackle these two issues and we establish a stability result for
the inverse problem of determining a rigid inclusion in an
anisotropic elastic plate by one boundary measurement.

{}From the mathematical point of view, \cite{Fi}, \cite{Gu}, we
shall work within the Kirchhoff-Love theory of thin, elastic
anisotropic plates under infinitesimal deformations. Let $\Omega$
denote the middle plane of the plate and let $h$ be its constant
thickness. We assume that $\Omega$ is a bounded domain of $\R^2$
of class $C^{1,1}$. The rigid inclusion $D$ is modelled as an open
simply connected domain compactly contained in $\Omega$. The
transversal displacement $w \in H^2(\Omega)$ of the plate
satisfies the following mixed boundary value problem
\begin{center}
\( {\displaystyle \left\{
\begin{array}{lr}
     {\rm div}({\rm div} (
      {\mathbb P}\nabla^2 w))=0,
      & \mathrm{in}\ \Omega \setminus \overline {D},
        \vspace{0.25em}\\
      ({\mathbb P} \nabla^2 w)n\cdot n=-\hat M_n, & \mathrm{on}\ \partial \Omega,
          \vspace{0.25em}\\
      {\rm div}({\mathbb P} \nabla^2 w)\cdot n+(({\mathbb P} \nabla^2
      w)n\cdot \tau),_s
      =(\hat M_\tau),_s, & \mathrm{on}\ \partial \Omega,
        \vspace{0.25em}\\
      w|_{\overline{D}} \in \mathcal{A}, &\mathrm{in}\ \overline{D},
          \vspace{0.25em}\\
        \frac{\partial w^e}{\partial n} = \frac{\partial w^i}{\partial n}, &\mathrm{on}\ \partial {D},
          \vspace{0.25em}\\
\end{array}
\right. } \) \vskip -8.9em
\begin{eqnarray}
& & \label{eq:1.dir-pbm-incl-rig-1}\\
& & \label{eq:1.dir-pbm-incl-rig-2}\\
& & \label{eq:1.dir-pbm-incl-rig-3}\\
& & \label{eq:1.dir-pbm-incl-rig-4}\\
& & \label{eq:1.dir-pbm-incl-rig-5}
\end{eqnarray}

\end{center}
coupled with the \emph{equilibrium conditions} for the rigid
inclusion $D$
\begin{multline}
  \label{eq:1.equil-rigid-incl}
    \int_{\partial D} \left ( {\rm div}({\mathbb P} \nabla^2 w)\cdot n+(({\mathbb P} \nabla^2
  w)n\cdot \tau),_s \right )g - (({\mathbb P} \nabla^2 w)n\cdot n)
  g_{,n} =0, \\   \quad \hbox{for every } g\in \mathcal{A},
\end{multline}
where $\mathcal{A}$ denotes the space of affine functions. In the
above equations, $n$ and $\tau$ are the unit outer normal and the
unit tangent vector to $\Omega \setminus \overline{D}$,
respectively, and we have defined $w^e \equiv w|_{\Omega \setminus
\overline{D}}$ and $w^i \equiv w|_{\overline{D}}$. Moreover,
$\widehat M_{\tau}$, $\widehat M_n$ are the twisting and bending
components of the assigned couple field $\widehat M$,
respectively. The plate tensor $\mathbb P$ is given by $\mathbb P
= \frac{h^3}{12}\mathbb C$, where $\mathbb C$ is the elasticity
tensor describing the response of the material.

Given any $\widehat{M}\in H^{-\frac{1}{2}}(\partial\Omega, \R^2)$,
satisfying the compatibility conditions $\int_{\partial \Omega}
\widehat{M}_i = 0$, for $i=1,2$, and for $\mathbb P$ bounded and
strongly convex, problem
\eqref{eq:1.dir-pbm-incl-rig-1}--\eqref{eq:1.equil-rigid-incl}
admits a solution $w \in H^{2}(\Omega)$, which is uniquely
determined up to addition of an affine function.

The uniqueness issue for the inverse problem under consideration is the following.

\textit{Given an open portion $\Sigma$ of $\partial\Omega$ and given two solutions $w_i$ to
\eqref{eq:1.dir-pbm-incl-rig-1}--\eqref{eq:1.equil-rigid-incl} for $D=D_i$, $i=1,2$,
satisfying
\begin{equation}
  \label{eq:1.cond_Dir1}
  w_1 = w_2, \hbox{ on } \Sigma,
\end{equation}
\begin{equation}
  \label{eq:1.cond_Dir2}
  \frac{\partial w_1}{\partial n} =\frac{\partial w_2}{\partial n} , \hbox{ on } \Sigma,
\end{equation}
does $D_1=D_2$ hold?}

Let us notice that, by \eqref{eq:1.dir-pbm-incl-rig-2},
\eqref{eq:1.dir-pbm-incl-rig-3}, \eqref{eq:1.cond_Dir1},
\eqref{eq:1.cond_Dir2}, $w_1$ and $w_2$ assume the same Cauchy
data on $\Sigma$. Under the a priori assumption of $C^{3,1}$
regularity of the boundary of the inclusion, a positive answer to
the uniqueness issue has been given in \cite{M-R2}, see also
\cite{M-R-V2} for a result of unique determination of a cavity in
an elastic plate by \textit{two} boundary measurements.

In order to face the stability issue, following the approach used
for the case of rigid inclusions in elastic bodies in \cite{M-R1},
we found it convenient to replace each solution $w_i$ introduced
above with $v_i=w_i-g_i$, where $g_i$ is the affine function which
coincides with $w_i$ on $\partial D_i$, $i=1,2$. By this approach,
maintaining the same letter to denote the solution, we rephrase
the equilibrium problem
\eqref{eq:1.dir-pbm-incl-rig-1}--\eqref{eq:1.dir-pbm-incl-rig-5}
in terms of the following mixed boundary value problem with
homogeneous Dirichlet conditions on the boundary of the rigid
inclusion

\begin{center}
\( {\displaystyle \left\{
\begin{array}{lr}
     {\rm div}({\rm div} (
      {\mathbb P}\nabla^2 w))=0,
      & \mathrm{in}\ \Omega \setminus \overline {D},
        \vspace{0.25em}\\
      ({\mathbb P} \nabla^2 w)n\cdot n=-\hat M_n, & \mathrm{on}\ \partial \Omega,
          \vspace{0.25em}\\
      {\rm div}({\mathbb P} \nabla^2 w)\cdot n+(({\mathbb P} \nabla^2
      w)n\cdot \tau),_s
      =(\hat M_\tau),_s, & \mathrm{on}\ \partial \Omega,
        \vspace{0.25em}\\
      w=0, &\mathrm{on}\ \partial D,
          \vspace{0.25em}\\
        \frac{\partial w}{\partial n} = 0, &\mathrm{on}\ \partial {D},
          \vspace{0.25em}\\
\end{array}
\right. } \) \vskip -8.9em
\begin{eqnarray}
& & \label{eq:1.dir-pbm-incl-rig-1bis}\\
& & \label{eq:1.dir-pbm-incl-rig-2bis}\\
& & \label{eq:1.dir-pbm-incl-rig-3bis}\\
& & \label{eq:1.dir-pbm-incl-rig-4bis}\\
& & \label{eq:1-dir-pbm-incl-rig-5bis}
\end{eqnarray}

\end{center}
\noindent which has  a unique solution $w\in H^2(\Omega\setminus
\overline{D})$.

On the other hand, it is clear that the arbitrariness of this
normalization, related to the fact that $g_i$ is unknown, $i=1,2$,
leads to the following formulation of the stability issue.

\textit{Given two solutions $w_i$ to
\eqref{eq:1.dir-pbm-incl-rig-1bis}--\eqref{eq:1-dir-pbm-incl-rig-5bis}, \eqref{eq:1.equil-rigid-incl} when $D=D_i$, $i=1,2$,
satisfying, for some $\epsilon>0$,
\begin{equation}
  \label{eq:cond_Dir_stab}
  \min_{g\in \mathcal{A}}\left\{ \| w_1 - w_2 - g \|_{L^2(\Sigma)} + \left\| \frac{\partial}{\partial n}(w_1 - w_2 - g) \right\|_{L^2(\Sigma)}\right\}\leq \epsilon,
\end{equation}
to evaluate the rate at which the Hausdorff distance between $D_1$
and $D_2$ tends to zero as $\epsilon$ tends to zero.}

In the present paper, assuming $C^{3,1}$ regularity of $\partial
D$, we prove the following constructive stability estimate of
log-log type
\begin{equation}
    \label{eq:log-log-estimate}
    d_{\cal H}( \overline{D_1},\overline{D_2} ) \leq C \left ( \log
    \left | \log \epsilon \right | \right )^{-\eta},
\end{equation}
where $C$, $\eta$, $C>0$ and $0<\eta\leq 1$, are constants only
depending on the a priori data, see Theorem \ref{theo:Main} for a
precise statement.

The methods used to prove \eqref{eq:log-log-estimate} are based
essentially on quantitative estimates of unique continuation
{}from Cauchy data for a solution of the mixed problem
\eqref{eq:1.dir-pbm-incl-rig-1bis}--\eqref{eq:1-dir-pbm-incl-rig-5bis},
\eqref{eq:1.equil-rigid-incl}. These estimates involve a
cornerstone result of unique continuation, namely the
\textit{Three Spheres Inequality} \eqref{eq:3sph} for solutions to
the plate equation \eqref{eq:1.dir-pbm-incl-rig-1bis}, which has
been determined in \cite{M-R-V5} under the very general assumption
that the elastic material of the plate obeys the \textit{dichotomy
condition} \eqref{3.D(x)bound}-\eqref{3.D(x)bound 2}.

The logarithmic character of the stability estimate
\eqref{eq:log-log-estimate} seems difficult to improve. First, the
Cauchy problem up to the boundary is severely ill-posed and, even
in the simpler context of the electrical impedance tomography
which involves a second order elliptic equation instead of a
fourth order one, the corresponding stability estimate with a
single logarithm is the best possible result, see \cite{A-B-R-V}
and also \cite{A-R-R-V} for a general discussion on the
ill-posedness of the Cauchy problem. We also quote \cite{DiC-R}
for examples of exponential instability for the inverse inclusion
problem in a conducting body. In addition, a further difficulty
arises in our analysis. It is due to the lack of quantitative
estimates of the \textit{strong unique continuation property at
the boundary} in the form of either \textit{Three Spheres
Inequality} or \textit{Doubling Inequality}. It has been shown in
\cite{A-B-R-V} that this is a key ingredient in proving that the
stability estimate for the corresponding inverse problem with
unknown boundaries in the conductivity context is not worse than
logarithm. This mathematical tool is not currently available for
the plate operator, even in the simplest case of isotropic
material, and this is the reason for the presence of a double
logarithm in the stability estimate \eqref{eq:log-log-estimate}.
We refer to \cite{M-R1} for a discussion on the analogous problem
in the determination of a rigid inclusion in an isotropic elastic
body {}by boundary measurements. Finally, as remarked in
\cite{M-R-V4}, it seems hopeless the possibility that solutions to
\eqref{eq:1.dir-pbm-incl-rig-1} can satisfy even a strong unique
continuation property at the interior, without any a priori
assumption on the anisotropy of the material, see also \cite{Ali}.
Regarding this point, our dichotomy condition
\eqref{3.D(x)bound}-\eqref{3.D(x)bound 2} basically contains the
same assumptions under which the unique continuation property
holds for a fourth order elliptic equation in two variables.

The paper is organized as follows. Some notation is presented in
Section $2$. In Section $3$ we state two auxiliary propositions
concerning the estimate of continuation {}from the interior
(Proposition \ref{prop:LPS}) and {}from Cauchy data (Proposition
\ref{prop:Cauchy1}), and we give the proof of the main Theorem
\ref{theo:Main}. Section $4$ contains the proofs of Proposition
\ref{prop:LPS} and Proposition \ref{prop:Cauchy1}. Proofs of
regularity estimates of the solution to the mixed problem
\eqref{eq:1.dir-pbm-incl-rig-1bis}--\eqref{eq:1-dir-pbm-incl-rig-5bis},
\eqref{eq:1.equil-rigid-incl} are presented in Section $5$ and, in
part, in the Appendix.

\section{Notation}

Let $P=(x_1(P), x_2(P))$ be a point of $\R^2$.
We shall denote by $B_r(P)$ the disk in $\R^2$ of radius $r$ and
center $P$ and by $R_{a,b}(P)$ the rectangle of center $P$ and sides parallel
to the coordinate axes, of length $2a$ and $2b$, namely
$R_{a,b}(P)=\{x=(x_1,x_2)\ |\ |x_1-x_1(P)|<a,\ |x_2-x_2(P)|<b \}$. To simplify the notation,
we shall denote $B_r=B_r(O)$, $R_{a,b}=R_{a,b}(O)$.

Given a bounded domain $\Omega$ in $\R^2$ we shall denote
\begin{equation}
  \label{eq:Omega_rho}
   \Omega_\rho=\{x\in \Omega\ |\ \hbox{dist}(x,\partial\Omega)>\rho\}.
\end{equation}

\noindent When representing locally a boundary as a graph, we use
the following definition.
\begin{definition}
  \label{def:2.1} (${C}^{k,1}$ regularity)
Let $\Omega$ be a bounded domain in ${\R}^{2}$. Given $k\in\N$, we say that a portion $S$ of
$\partial \Omega$ is of \textit{class ${C}^{k,1}$ with
constants $\rho_{0}$, $M_{0}>0$}, if, for any $P \in S$, there
exists a rigid transformation of coordinates under which we have
$P=0$ and
\begin{equation*}
  \Omega \cap R_{\frac{\rho_0}{M_0},\rho_0}=\{x=(x_1,x_2) \in R_{\frac{\rho_0}{M_0},\rho_0}\quad | \quad
x_{2}>\psi(x_1)
  \},
\end{equation*}
where $\psi$ is a ${C}^{k,1}$ function on
$\left(-\frac{\rho_0}{M_0},\frac{\rho_0}{M_0}\right)$ satisfying
\begin{equation*}
\psi(0)=0,
\end{equation*}
\begin{equation*}
\psi' (0)=0, \quad \hbox {when } k \geq 1,
\end{equation*}
\begin{equation*}
\|\psi\|_{{C}^{k,1}\left(-\frac{\rho_0}{M_0},\frac{\rho_0}{M_0}\right)} \leq M_{0}\rho_{0}.
\end{equation*}

\medskip
\noindent When $k=0$ we also say that $S$ is of
\textit{Lipschitz class with constants $\rho_{0}$, $M_{0}$}.
\end{definition}
\begin{rem}
  \label{rem:2.1}
  We use the convention to normalize all norms in such a way that their
  terms are dimensionally homogeneous with the $L^\infty$ norm and coincide with the
  standard definition when the dimensional parameter equals one.
  For instance, the norm appearing above is meant as follows
\begin{equation*}
  \|\psi\|_{{C}^{k,1}\left(-\frac{\rho_0}{M_0},\frac{\rho_0}{M_0}\right)} =
  \sum_{i=0}^{k+1} \rho_0^i
  \|\psi^{(i)}\|_{{L}^{\infty}\left(-\frac{\rho_0}{M_0},\frac{\rho_0}{M_0}\right)},
\end{equation*}
where $\psi^{(i)}$ is the $i$-th derivative with respect to the $x_1$ variable.
Similarly, denoting by $\nabla^i u$ the vector which components are the derivatives of order $i$ of a
function $u$ defined in $\Omega$, we denote
\begin{equation*}
  \|u\|_{{C}^{k,1}(\Omega)} =\sum_{i=0}^{k+1}
  {\rho_{0}}^{i}\|{\nabla}^{i} u\|_{{L}^{\infty}(\Omega)},
\end{equation*}
\begin{equation*}
\|u\|_{L^2(\Omega)}=\rho_0^{-1}\left(\int_\Omega
u^2\right) ^{\frac{1}{2}},
\end{equation*}
\begin{equation*}
\|u\|_{H^m(\Omega)}=\rho_0^{-1}\left(\sum_{i=0}^m
\rho_0^{2i}\int_\Omega|\nabla^i u|^2\right)^{\frac{1}{2}},
\end{equation*}
and so on for boundary and trace norms such as
$\|\cdot\|_{H^{\frac{1}{2}}(\partial\Omega)}$,
$\|\cdot\|_{H^{-\frac{1}{2}}(\partial\Omega)}$.

Notice also that, when $\Omega=B_{R}$, then $\Omega$ satisfies
Definition \ref{def:2.1} with $\rho_{0}=R$, $M_0=2$ and therefore, for
instance, we have
\begin{equation*}
\|u\|_{H^m(B_R)}=R^{-1}\left(\sum_{i=0}^m
R^{2i}\int_{B_R}|\nabla^i u|^2\right)^{\frac{1}{2}},
\end{equation*}
\end{rem}

Given a bounded domain $\Omega$ in $\R^2$ such that $\partial
\Omega$ is of class $C^{k,1}$, with $k\geq 1$, we consider as
positive the orientation of the boundary induced by the outer unit
normal $n$ in the following sense. Given a point
$P\in\partial\Omega$, let us denote by $\tau=\tau(P)$ the unit
tangent at the boundary in $P$ obtained by applying to $n$ a
counterclockwise rotation of angle $\frac{\pi}{2}$, that is
\begin{equation}
    \label{eq:2.tangent}
        \tau=e_3 \times n,
\end{equation}
where $\times$ denotes the vector product in $\R^3$, $\{e_1,
e_2\}$ is the canonical basis in $\R^2$ and $e_3=e_1 \times e_2$.

Given any connected component $\cal C$ of $\partial \Omega$ and
fixed a point $P\in\cal C$, let us define as positive the
orientation of $\cal C$ associated to an arclength
parametrization $\varphi(s)=(x_1(s), x_2(s))$, $s \in [0, l(\cal
C)]$, such that $\varphi(0)=P$ and
$\varphi'(s)=\tau(\varphi(s))$, where $l(\cal C)$ denotes the
length of $\cal C$.

Throughout the paper, we denote by $\partial_i u$, $\partial_s u$, and $\partial_n u$
the derivatives of a function $u$ with respect to the $x_i$
variable, to the arclength $s$ and to the normal direction $n$,
respectively, and similarly for higher order derivatives.

We denote by $\mathbb{M}^2$ the space of $2 \times 2$ real valued
matrices and by ${\mathcal L} (X, Y)$ the space of bounded linear
operators between Banach spaces $X$ and $Y$.

For every $2 \times 2$ matrices $A$, $B$ and for every $\mathbb{L}
\in{\mathcal L} ({\mathbb{M}}^{2}, {\mathbb{M}}^{2})$, we use the
following notation:
\begin{equation}
  \label{eq:2.notation_1}
  ({\mathbb{L}}A)_{ij} = L_{ijkl}A_{kl},
\end{equation}
\begin{equation}
  \label{eq:2.notation_2}
  A \cdot B = A_{ij}B_{ij},
\end{equation}
\begin{equation}
  \label{eq:2.notation_3}
  |A|= (A \cdot A)^{\frac {1} {2}},
\end{equation}
\begin{equation}
  \label{eq:2.notation_3bis}
  A^{sym} =  \frac{1}{2} \left ( A + A^t \right ),
\end{equation}
where $A^t$ denotes the transpose of the matrix $A$.
Notice that here and in the sequel summation over repeated indexes
is implied.

Finally, let us introduce the linear space of the affine functions
on $\R^2$
\begin{equation*}
  \label{eq:2.notation_3ter}
  \mathcal{A}=\{g(x_1,x_2)=ax_1+bx_2+c,\  a,b,c \in\R\}.
\end{equation*}

\subsection{A priori information}
\medskip
\noindent {\it i) A priori information on the domain.}

Let us consider a thin plate
$\Omega\times[-\frac{h}{2},\frac{h}{2}]$ with middle surface
represented by a bounded domain $\Omega$ in $\R^2$ and having
uniform thickness $h$, $h<<\hbox{diam}(\Omega)$.

We shall assume
that, given $\rho_0$, $M_1>0$,
\begin{equation}
   \label{eq:bound_area}
|\Omega|\leq M_1\rho_0^2,
\end{equation}
where $|\Omega|$ denotes the Lebesgue measure of $\Omega$. We shall also assume
that $\Omega$ contains an open simply connected rigid
inclusion $D$ such that
\begin{equation}
   \label{eq:compactness}
    \hbox{dist}(D, \partial \Omega) \geq \rho_0.
\end{equation}
Moreover, we denote by $\Sigma$ an open portion
within $\partial \Omega$ representing the part of the boundary
where measurements are taken.

Concerning the regularity of the boundaries, given $M_0>0$, we assume that
\begin{equation}
   \label{eq:reg_Omega}
\partial\Omega \hbox{ is of } class\  C^{2,1}
\ with\ constants\ \rho_0, M_0,
\end{equation}
\begin{equation}
   \label{eq:reg_Sigma}
\Sigma \hbox{ is of } class\  C^{3,1} \ with\ constants\
\rho_0, M_0.
\end{equation}
\begin{equation}
   \label{eq:reg_D}
\partial D \hbox{ is of } class\  C^{3,1}
\ with\ constants\ \rho_0, M_0.
\end{equation}

Moreover, we shall assume that for some $P_0\in\Sigma$ and some $\delta_0$, $0<\delta_0<1$,
\begin{equation}
   \label{eq:large_enough}
   \partial\Omega\cap
R_{\frac{\rho_0}{M_0},\rho_0}(P_0)\subset\Sigma,
\end{equation}
and that
\begin{equation}
   \label{eq:small_enough}
   |\Sigma|\leq(1-\delta_0)|\partial\Omega|.
\end{equation}
\medskip
\noindent {\it ii) Assumptions about the boundary data.}

On the Neumann data $\widehat{M}$ we assume that
\begin{equation}
   \label{eq:reg_M}
\widehat{M}\in L^2(\partial \Omega,\R^2),\quad
(\widehat{M}_n,(\widehat{M}_\tau),_s)\not\equiv 0,
\end{equation}
\begin{equation}
   \label{eq:supp_M}
\hbox{supp}(\widehat{M})\subset\subset\Sigma,
\end{equation}
the (obvious) compatibility condition
\begin{equation}
   \label{eq:M_comp}
    \int_{\partial \Omega} \widehat{M}_i = 0, \quad i=1,2,
\end{equation}
and that, for a given constant $F>0$,
\begin{equation}
\label{eq:M_frequency}
   \frac{\|\widehat{M}\|_{L^2(\partial
   \Omega ,\R^2)}}{ \|\widehat{M}\|_{H^{-\frac{1}{2}}(\partial \Omega,\R^2)}}\leq
   F.
\end{equation}

{\it iii) Assumptions about the elasticity tensor.}

Let us assume that the plate is made of nonhomogeneous linear
elastic material with plate tensor
\begin{equation}
\label{eq:P_def}
   \mathbb{P}=\frac{h^3}{12}\mathbb{C},
\end{equation}
where the elasticity tensor $\C(x) \in{\cal L}
({\M}^{2}, {\M}^{2})$ has cartesian components
$C_{ijkl}$ which satisfy the following symmetry conditions
\begin{equation}
  \label{eq:sym-conditions-C-components}
    C_{ijkl} = C_{ klij} =
    C_{ klji} \quad i,j,k,l
    =1,2, \hbox{ a.e. in } \Omega.
\end{equation}
We recall that
\eqref{eq:sym-conditions-C-components} are equivalent to
\begin{equation}
  \label{eq:sym-conditions-C-1}
  {\C}A={\C} {A}^{sym},
\end{equation}
\begin{equation}
  \label{eq:eq:sym-conditions-C-2}
  {\C}A \quad \hbox{is } symmetric,
\end{equation}
\begin{equation}
  \label{eq:eq:sym-conditions-C-3}
  {\C}A \cdot B= {\C}B \cdot A,
\end{equation}
for every $2 \times 2$ matrices $A$, $B$.

In order to simplify the presentation, we shall assume that the tensor
$\mathbb{C}$ is defined in all of $\R^2$.

Condition \eqref{eq:sym-conditions-C-components} implies that instead of $16$ coefficients
we actually deal with $6$ coefficients and we denote
\begin{center}
\( {\displaystyle \left\{
\begin{array}{lr}
    C_{1111}=A_0, \ \ C_{1122}=C_{2211}=B_0,
         \vspace{0.12em}\\
    C_{1112}=C_{1121}=C_{1211}=C_{2111}=C_0,
        \vspace{0.12em}\\
    C_{2212}=C_{2221}=C_{1222}=C_{2122}=D_0,
        \vspace{0.12em}\\
    C_{1212}=C_{1221}=C_{2112}=C_{2121}=E_0,
        \vspace{0.12em}\\
    C_{2222}=F_0,
        \vspace{0.25em}\\
\end{array}
\right. } \) \vskip -3.0em
\begin{eqnarray}
\ & & \label{3.coeff6}
\end{eqnarray}
\end{center}
and
\begin{equation}
    \label{3.coeffsmall}
    a_0=A_0, \ a_1=4C_0, \ a_2=2B_0+4E_0, \ a_3=4D_0, \ a_4=F_0.
\end{equation}

Let $S(x)$ be the following $7\times 7$ matrix
\begin{equation}
    \label{3. S(x)}
    S(x) = {\left(
\begin{array}{ccccccc}
  a_0   & a_1   & a_2   & a_3   & a_4   & 0    &    0    \\
  0     & a_0   & a_1   & a_2   & a_3   & a_4  &    0    \\
  0     & 0     & a_0   & a_1   & a_2   & a_3  &    a_4  \\
  4a_0  & 3a_1  & 2a_2  & a_3   & 0     & 0    &    0    \\
  0     & 4a_0  & 3a_1  & 2a_2  & a_3   & 0    &    0    \\
  0     & 0     & 4a_0  & 3a_1  & 2a_2  & a_3  &    0    \\
  0     & 0     & 0     & 4a_0  & 3a_1  & 2a_2 &    a_3  \\
  \end{array}
\right)},
\end{equation}
and
\begin{equation}
    \label{3.D(x)}
    {\mathcal{D}}(x)= \frac{1}{a_0} |\det S(x)|.
\end{equation}

On the elasticity tensor $\mathbb{C}$ we make the following
assumptions:

\medskip
{I)} \textit{Regularity}
\begin{equation}
  \label{eq:3.bound}
  \mathbb{C} \in C^{1,1}(\R^2,   {\mathcal L} ({\mathbb{M}}^{2},
  {\mathbb{M}}^{2})),
\end{equation}
with
\begin{equation}
  \label{eq:3.bound_quantit}
  \sum_{i,j,k,l=1}^2 \sum_{m=0}^2 \rho_0^m \|\nabla^m C_{ijkl}\|_{L^\infty(\R^2)} \leq
    M,
\end{equation}
where $M$ is a positive constant;

{II)} \textit{Ellipticity (strong convexity)} There exists $\gamma>0$
such that

\begin{equation}
  \label{eq:3.convex}
    {\mathbb{C}}A \cdot A \geq  \gamma |A|^2, \qquad \hbox{in } \R^2,
\end{equation}
for every $2\times 2$ symmetric matrix $A$.

{III)} \textit{Dichotomy condition}
\begin{subequations}
\begin{eqnarray}
\label{3.D(x)bound} either &&  {\mathcal{D}}(x)>0,
\quad\hbox{for every } x\in \R^2, \\[2mm]
\label{3.D(x)bound 2} or && {\mathcal{D}}(x)=0, \quad\hbox{for every }
x\in \R^2,
\end{eqnarray}
\end{subequations}
where ${\mathcal{D}}(x)$ is defined by \eqref{3.D(x)}.

\begin{rem}
  \label{rem:dichotomy} Whenever \eqref{3.D(x)bound} holds we denote
\begin{equation}
    \label{delta-1}
    \delta_1=\min_{\R^2}{\mathcal{D}}.
\end{equation}
We emphasize that, in all the following statements, whenever a
constant is said to depend on $\delta_1$ (among other quantities) it
is understood that such dependence occurs \textit{only} when
\eqref{3.D(x)bound} holds.
\end{rem}

We shall refer to the set of constants $M_0$, $M_1$, $\delta_0$, $F$, $\gamma$, $M$, $\delta_1$ as the \emph{a priori data}.
The dependence on the thickness parameter $h$ will be omitted.

In the sequel we shall consider the following boundary value problem of mixed type

\begin{center}
\( {\displaystyle \left\{
\begin{array}{lr}
     {\rm div}({\rm div} (
      {\mathbb P}\nabla^2 w))=0,
      & \mathrm{in}\ \Omega \setminus \overline {D},
        \vspace{0.25em}\\
      ({\mathbb P} \nabla^2 w)n\cdot n=-\hat M_n, & \mathrm{on}\ \partial \Omega,
          \vspace{0.25em}\\
      {\rm div}({\mathbb P} \nabla^2 w)\cdot n+(({\mathbb P} \nabla^2
      w)n\cdot \tau),_s
      =(\hat M_\tau),_s, & \mathrm{on}\ \partial \Omega,
        \vspace{0.25em}\\
      w=0, &\mathrm{on}\ \partial D,
          \vspace{0.25em}\\
        \frac{\partial w}{\partial n} = 0, &\mathrm{on}\ \partial {D},
          \vspace{0.25em}\\
\end{array}
\right. } \) \vskip -8.9em
\begin{eqnarray}
& & \label{eq:dir-pbm-incl-rig-1}\\
& & \label{eq:dir-pbm-incl-rig-2}\\
& & \label{eq:dir-pbm-incl-rig-3}\\
& & \label{eq:dir-pbm-incl-rig-4}\\
& & \label{eq:dir-pbm-incl-rig-5}
\end{eqnarray}

\end{center}
coupled with the \emph{equilibrium conditions} for the rigid
inclusion $D$
\begin{multline}
  \label{eq:equil-rigid-incl}
    \int_{\partial D} \left ( {\rm div}({\mathbb P} \nabla^2 w)\cdot n+(({\mathbb P} \nabla^2
  w)n\cdot \tau),_s \right )g - (({\mathbb P} \nabla^2 w)n\cdot n)
  g_{,n} =0, \\   \quad \hbox{for every } g\in \mathcal{A}.
\end{multline}
By standard variational arguments, it is easy to see that problem
\eqref{eq:dir-pbm-incl-rig-1}--\eqref{eq:equil-rigid-incl} admits a unique solution
$w\in H^2(\Omega\setminus \overline{D})$ such that
\begin{equation}
    \label{eq:w_stima_H2}
    \|w\|_{H^2(\Omega\setminus \overline{D})}\leq C \rho_0^2\|\widehat{M}\|_{H^{-\frac{1}{2}}(\partial\Omega,\R^2)},
\end{equation}
where $C>0$ only depends on $\gamma$, $M_0$, and $M_1$.

\section{Statement and proof of the main result}

Here and in the sequel we shall denote by $G$ the connected
component of $\Omega \setminus   \overline{(D_1 \cup D_2)}$ such
that $\Sigma \subset \partial{G}$.

\begin{theo}[Stability result]
  \label{theo:Main}
Let $\Omega$ be a bounded domain in $\R^2$ satisfying \eqref{eq:bound_area} and \eqref{eq:reg_Omega}.
Let $D_i$, $i=1,2$, be two simply connected open
subsets of $\Omega$ satisfying \eqref{eq:compactness}
and \eqref{eq:reg_D}.
Moreover, let $\Sigma$ be an open portion of $\partial\Omega$
satisfying \eqref{eq:reg_Sigma}, \eqref{eq:large_enough} and \eqref{eq:small_enough}.
Let $\widehat{M}\in L^2(\partial\Omega,\R^2)$ satisfy
\eqref{eq:reg_M}--\eqref{eq:M_frequency} and let the plate tensor $\mathbb{P}$
given by \eqref{eq:P_def} satisfy the symmetry conditions \eqref{eq:sym-conditions-C-components},
the regularity condition \eqref{eq:3.bound_quantit},
the strong convexity condition \eqref{eq:3.convex} and the dichotomy condition.
Let $w_i\in
H^2(\Omega \setminus \overline{D_i})$ be the solution to
\eqref{eq:dir-pbm-incl-rig-1}--\eqref{eq:equil-rigid-incl}, when $D=D_i$, $i=1,2$. If,
given $\epsilon>0$, we have
\begin{equation}
    \label{eq:small_L2}
\min_{g \in
\cal{A}} \left\{\|w_1 - w_2 -g \|_{L^2(\Sigma)}+\rho_0
\left\|\frac{\partial}{\partial n}
(w_1 - w_2 -g) \right\|_{L^2(\Sigma)}\right\}\leq
\epsilon,
\end{equation}
then we have
\begin{equation}
    \label{eq:small_Haus_bound}
d_{\cal H}(\partial D_1,\partial D_2) \leq
\rho_0\omega\left(\frac{\epsilon}
{\rho_0^2\|\widehat{M}\|_{H^{-\frac{1}{2}}(\partial
\Omega,\R^2)}}\right)
\end{equation}
and
\begin{equation}
    \label{eq:small_Haus_inclusion}
d_{\cal H}( \overline{D_1},\overline{D_2} ) \leq
\rho_0\omega\left(\frac{\epsilon}
{\rho_0^2\|\widehat{M}\|_{H^{-\frac{1}{2}}(\partial
\Omega,\R^2)}}\right),
\end{equation}
where $\omega$ is an increasing continuous function on
$[0,\infty)$ which satisfies
\begin{equation}
    \label{eq:omega_loglog}
\omega(t)\leq C(\log|\log t|)^{-\eta}, \quad\hbox{for every }t, \
0<t<e^{-1},
\end{equation}
and $C$, $\eta$, $C>0$, $0<\eta\leq 1$, are constants only
depending on the a priori data.
\end{theo}

The proof of Theorem \ref{theo:Main} is obtained {from} the
following sequence of Propositions.

\begin{prop}[Lipschitz Propagation of Smallness]
  \label{prop:LPS}
Let $\Omega$ be a bounded domain in $\R^2$ satisfying \eqref{eq:bound_area} and \eqref{eq:reg_Omega}.
Let $D$ be an open simply connected subset of $\Omega$
satisfying \eqref{eq:compactness}, \eqref{eq:reg_D}. Let $w\in
H^2(\Omega \setminus \overline{D})$ be the solution to
\eqref{eq:dir-pbm-incl-rig-1}--\eqref{eq:equil-rigid-incl}, where the plate
tensor $\mathbb{P}$
given by \eqref{eq:P_def} satisfies \eqref{eq:sym-conditions-C-components},
\eqref{eq:3.bound_quantit},
\eqref{eq:3.convex} and the dichotomy condition.
Let the couple field $\widehat{M}$ satisfy
\eqref{eq:reg_M}--\eqref{eq:M_frequency}.

There exists $s>1$, only depending on $\gamma$, $M$, $\delta_1$, $M_0$
and $\delta_0$, such that for every $\rho>0$ and every $\bar x\in
(\Omega\setminus \overline{D})_{s\rho}$, we have
\begin{equation}
   \label{eq:LPS}
\int_{B_\rho(\bar x)}|\nabla^2 w|^2\geq
\frac{C\rho_0^2}{\exp\left[A\left(\frac{\rho_0}{\rho}\right)^B\right]}
\| \widehat{M} \|_{H^{-\frac{1}{2}}(\partial\Omega,\R^2)}^2,
\end{equation}
where $A>0$, $B>0$ and $C>0$ only depend on $\gamma$, $M$, $\delta_1$, $M_0$, $M_1$, $\delta_0$ and $F$.
\end{prop}

\begin{prop}[Stability Estimate of Continuation
{from} Cauchy Data]
  \label{prop:Cauchy1}
Let the hypotheses of Theorem \ref{theo:Main} be satisfied. We
have
\begin{equation}
   \label{eq:Cauchy1}
\int_{D_2\setminus \overline{D_1}}|\nabla^2 w_1|^2\leq
\rho_0^2\|\widehat{M}\|_{H^{-\frac{1}{2}}(\partial \Omega
,\R^2)}^2\omega\left(\frac{\epsilon}
{\rho_0^2\|\widehat{M}\|_{H^{-\frac{1}{2}}(\partial
\Omega ,\R^2)}}\right),
\end{equation}
\begin{equation}
   \label{eq:Cauchy1bis}
\int_{D_1\setminus \overline{D_2}}|\nabla^2 w_2|^2\leq
\rho_0^2\|\widehat{M}\|_{H^{-\frac{1}{2}}(\partial \Omega
,\R^2)}^2\omega\left(\frac{\epsilon}
{\rho_0^2\|\widehat{M}\|_{H^{-\frac{1}{2}}(\partial
\Omega ,\R^2)}}\right),
\end{equation}
where $\omega$ is an increasing continuous function on
$[0,\infty)$ which satisfies
\begin{equation}
   \label{eq:omega_cauchy_loglog}
\omega(t)\leq C(\log|\log t|)^{-\frac{1}{2}},\qquad \hbox{for every }
t<e^{-1},
\end{equation}
with $C>0$ only depending on $\gamma$, $M$, $\delta_1$,
$M_0$, $M_1$ and $\delta_0$.

If we assume, in addition, that there exist $L>0$ and $\tilde
\rho_0$, $0<\tilde \rho_0\leq \rho_0$, such that $\partial G$ is
of {\it Lipschitz class with constants} $\tilde \rho_0$, $L$, then
\eqref{eq:Cauchy1}--\eqref{eq:Cauchy1bis} hold with $\omega$
given by
\begin{equation}
   \label{eq:omega_cauchy_log}
\omega(t)\leq C|\log t|^{-\sigma},\qquad \hbox{\rm for  every }
t<1,
\end{equation}
where $\sigma>0$ and $C>0$ only depend on $\gamma$, $M$, $\delta_1$,
$M_0$, $M_1$, $\delta_0$, $L$ and $\frac{\tilde
\rho_0}{\rho_0}$.
\end{prop}
\begin{proof} [Proof of Theorem ~\ref{theo:Main}]
Let us denote, for simplicity, $d=d_{\cal H}(\partial D_1,\partial
D_2)$. Let us see that, if $\eta>0$ is such that
\begin{equation}
   \label{eq:eta}
\int_{D_2 \setminus \overline{D_1}}|\nabla^2 w_1|^2 \leq\ \frac{\eta}{\rho_0^2}, \quad
\quad \int_{D_1 \setminus \overline{D_2}}|\nabla^2 w_2|^2 \leq \frac{\eta}{\rho_0^2},
\end{equation}
then we have
\begin{equation}
   \label{eq:d<eta}
d\leq
C\rho_0\left[\log\left(\frac{C\rho_0^4\|\widehat{M}\|_{H^{-\frac{1}{2}}(\partial
\Omega ,\R^n)}^2} {\eta}\right)\right]^{-\frac{1}{B}},
\end{equation}
where $B>0$ and $C>0$ only depend on $\gamma$, $M$, $\delta_1$,
$M_0$, $M_1$, $\delta_0$ and $F$.

We may assume, with no loss of generality, that there exists
$x_0\in \partial D_1$ such that $\hbox{dist}(x_0,\partial D_2)=d$.
Let us distinguish two cases:

\item{i)} $B_d(x_0) \subset D_2$;
\item{ii)} $B_d(x_0) \cap D_2 = \emptyset$.

In case i), by the regularity assumptions made on $\partial D_1$,
there exists $x_1\in D_2 \setminus D_1$ such that
$B_{td}(x_1)\subset D_2 \setminus D_1$, with
$t=\frac{1}{1+\sqrt{1+M_0^2}}$.

By \eqref{eq:eta} and by Proposition \ref{prop:LPS} with
$\rho=\frac{td}{s}$, we have
\begin{equation}
   \label{eq:eta>}
\eta\geq\frac{C\rho_0^4}{\exp{\left[A\left(\frac{s\rho_0}{td}\right)^B\right]}}
\|\widehat{M}\|_{H^{-\frac{1}{2}}(\partial \Omega ,\R^2)}^2,
   \end{equation}
where $A>0$, $B>0$ and $C>0$ only depend on $\gamma$, $M$, $\delta_1$,
$M_0$, $M_1$, $\delta_0$ and $F$.

By \eqref{eq:eta>} we easily find \eqref{eq:d<eta}.

Case ii) can be treated similarly by substituting $w_1$ with
$w_2$.

Hence, by Proposition \ref{prop:Cauchy1} and assuming $\epsilon <
e^{-e} \rho_0^2
\|\widehat{M}\|_{H^{-\frac{1}{2}}(\partial \Omega ,\R^2)}$ we obtain
\begin{equation}
   \label{eq:d<logloglog}
d\leq
C\rho_0\left\{\log\left[\log\left|\log\frac{\epsilon}{\rho_0^2
\|\widehat{M}\|_{H^{-\frac{1}{2}}(\partial \Omega,\R^2)}}\right|
\right]\right\}^{-\frac{1}{B}},
\end{equation}
where $B>0$ and $C>0$ only depend on $\gamma$, $M$, $\delta_1$,
$M_0$, $M_1$, $\delta_0$ and $F$.

Thus we have obtained a stability estimate of $\log$-$\log$-$\log$
type.

In order to prove an analogous estimate for the Hausdorff distance
between $\overline{D_1}$ and $\overline{D_2}$, let us now set
$d=d_{\cal H}(\overline{D_1},
\overline{D_2})$ and let us assume, with no loss of
generality, that there exists $x_0\in\overline{D_1}$ such that
$\hbox{dist}(x_0,\overline{D_2})=d$. If $B_d(x_0)\subset D_1$,
then $B_d(x_0)\subset D_1\setminus D_2$ and \eqref{eq:eta>}
follows with $t$ replaced by $1$. If, otherwise, $B_d(x_0)
\nsubseteq D_1$, then $\hbox{dist}(x_0,\partial D_1)\leq d$, and
we can distinguish two cases:

\item{i)} $\hbox{dist}(x_0,\partial D_1)>\frac{d}{2}$,
\item{ii)} $\hbox{dist}(x_0,\partial D_1)\leq\frac{d}{2}$.

When i) holds, then $B_{\frac{d}{2}}(x_0)\subset D_1\setminus D_2$
and again \eqref{eq:eta>} follows with $t$ replaced by
$\frac{1}{2}$. When ii) holds, there exists $y_0\in\partial D_1$
such that $|y_0-x_0| \leq \frac{d}{2}$. Therefore there exists
$y_1\in D_1$ such that $B_{\frac{td}{2}}(y_1)\subset D_1\setminus
D_2$, and \eqref{eq:eta>} follows with $t$ replaced by
$\frac{t}{2}$. From \eqref{eq:eta>}, arguing as above, we obtain \eqref{eq:d<logloglog}
for $d=d_{\cal H}(\overline{D_1},
\overline{D_2})$.

Next, by this rough estimate, we can apply Proposition 3.6 in \cite{A-B-R-V}  which ensure that we can find $\epsilon_0>0$,
only depending on $\gamma$, $M$, $\delta_1$,
$M_0$, $M_1$, $\delta_0$ and $F$, such that if $\epsilon\leq \epsilon_0$ then
$\partial G$ is of Lipschitz class with constants $\widetilde{\rho_0}$, $L$, with
$L$ and $\frac{\widetilde{\rho_0}}{\rho_0}$ only depending on $M_0$.
By the second part of
Proposition \ref{prop:Cauchy1}, the $\log$-$\log$
type estimates \eqref{eq:small_Haus_bound}, \eqref{eq:small_Haus_inclusion},
\eqref{eq:omega_loglog} follow.
\end{proof}
\section{Proofs of Propositions \ref{prop:LPS} and \ref{prop:Cauchy1}}.

\noindent
We need to premise some auxiliary results. The first one is the basic tool
of our approach, the three spheres inequality.

\begin{prop} [Three Spheres Inequality (\cite{M-R-V5}, Proposition 5.1)]
\label{prop:3spheres} Let $\Omega$ be a domain in $\R^2$, and let the plate tensor $\mathbb{P}$
given by \eqref{eq:P_def} satisfy \eqref{eq:sym-conditions-C-components},
\eqref{eq:3.bound_quantit},
\eqref{eq:3.convex} and the dichotomy condition. Let $w\in
H^2(\Omega)$ be a weak solution to the equation
\begin{equation}
  \label{eq:plate_eq}
{\rm div}({\rm div} (
      {\mathbb P}\nabla^2 w))=0,
       \quad \hbox{in }\Omega.
\end{equation}

For every $r_1, r_2, r_3, \overline{r}$, $0<r_1<r_2<r_3\leq \overline{r}$, and
for every $x\in \Omega_{\overline{r}}$ we have
\begin{equation}
  \label{eq:3sph}
   \int_{B_{r_{2}}(x)}|\nabla ^2w|^{2} \leq C
   \left(  \int_{B_{r_{1}}(x)}|\nabla ^2w|^{2}
   \right)^{\delta}\left(  \int_{B_{r_{3}}(x)}|\nabla ^2w|^{2}
   \right)
   ^{1-\delta},
\end{equation}
where $C>0$ and $\delta$, $0<\delta<1$, only depend on
$\gamma$, $M$, $\delta_1$, $\frac{r_{3}}{r_{2}}$ and
$\frac{r_{3}}{r_{1}}$.
\end{prop}

\medskip
\noindent

Let us define, for $\rho\leq\rho_0$,

\begin{equation}
  \label{eq:intornoU}
   \mathcal{U}^\rho=\{x\in \Omega\ |\ \hbox{dist}(x,\partial\Omega)\leq\rho\},
\end{equation}
that is $\mathcal{U}^\rho=\Omega\setminus \Omega_\rho$.

\medskip
\noindent

The following three lemmas state regularity estimates.

\begin{lem}[Global $H^3$ regularity for the mixed problem]
  \label{lem:H^3}
Let $\Omega$ be a bounded domain in $\R^2$ satisfying
\eqref{eq:bound_area} and \eqref{eq:reg_Omega}. Let $D$ be a simply connected open
subset of $\Omega$ satisfying \eqref{eq:compactness} and \eqref{eq:reg_D}.
Let $\widehat{M}\in H^{\frac{1}{2}}(\partial\Omega,\R^2)$ satisfy
\eqref{eq:M_comp}. Let the plate tensor $\mathbb{P}$ be defined by
\eqref{eq:P_def} and satisfying
\eqref{eq:sym-conditions-C-components}, \eqref{eq:3.bound_quantit}, \eqref{eq:3.convex}. Let
$w\in H^2(\Omega \setminus \overline{D})$ be the solution to
\eqref{eq:dir-pbm-incl-rig-1}--\eqref{eq:equil-rigid-incl}. We
have
\begin{equation}
  \label{eq:reg_H3}
   \|w\|_{H^3(\Omega\setminus \overline{D})}\leq
   C\rho_0^2\|\widehat{M}\|_{H^{\frac{1}{2}}(\partial\Omega,\R^2)},
\end{equation}
where $C>0$ only depends on $M_0$, $M_1$, $\gamma$, $M$.
\end{lem}
\begin{lem}[$H^4$ regularity up to the boundary of the rigid inclusion]
  \label{lem:H^4}
Let $\Omega$ be a bounded domain in $\R^2$ satisfying
\eqref{eq:bound_area} and \eqref{eq:reg_Omega}. Let $D$ be a
simply connected open subset of $\Omega$ satisfying
\eqref{eq:compactness} and \eqref{eq:reg_D} Let $\widehat{M}\in
H^{-\frac{1}{2}}(\partial\Omega,\R^2)$ satisfy \eqref{eq:M_comp}.
Let the plate tensor $\mathbb{P}$ be defined by \eqref{eq:P_def}
and satisfying \eqref{eq:sym-conditions-C-components},
\eqref{eq:3.bound_quantit}, \eqref{eq:3.convex}. Let $w\in
H^2(\Omega \setminus \overline{D})$ be the solution to
\eqref{eq:dir-pbm-incl-rig-1}--\eqref{eq:equil-rigid-incl}.

We have
\begin{equation}
  \label{eq:reg_H4}
   \|w\|_{H^4\left((\Omega\setminus \overline{D})\setminus \mathcal{U}^{\frac{\rho_0}{8}}\right)}
   \leq
   C\rho_0^2\|\widehat{M}\|_{H^{-\frac{1}{2}}(\partial\Omega,\R^2)},
\end{equation}
where $C>0$ only depends on $M_0$, $M_1$, $\gamma$, $M$.
\end{lem}
\begin{lem}[$C^{2,\alpha}$ regularity up to the boundary of the rigid inclusion]
  \label{lem:C^1,1}
Let $\Omega$ be a bounded domain in $\R^2$ satisfying
\eqref{eq:bound_area} and \eqref{eq:reg_Omega}. Let $D$ be a
simply connected open subset of $\Omega$ satisfying
\eqref{eq:compactness} and \eqref{eq:reg_D}. Let $\widehat{M}\in
H^{-\frac{1}{2}}(\partial\Omega,\R^2)$ satisfy \eqref{eq:M_comp}
and let the plate tensor $\mathbb{P}$ given by \eqref{eq:P_def}
satisfy \eqref{eq:sym-conditions-C-components},
\eqref{eq:3.bound_quantit} and \eqref{eq:3.convex}. Let $w\in
H^2(\Omega \setminus \overline{D})$ be the solution to
\eqref{eq:dir-pbm-incl-rig-1}--\eqref{eq:equil-rigid-incl}. We
have
\begin{equation}
  \label{eq:reg_classica}
   \|w\|_{C^{2,\alpha}\left((\Omega\setminus \overline{D})\setminus \mathcal{U}^{\frac{\rho_0}{8}}\right)}
   \leq C\rho_0^2\|\widehat{M}\|_{H^{-\frac{1}{2}}(\partial\Omega,\R^2)}, \quad\hbox{for every }\alpha<1,
\end{equation}
where $C>0$ only depends on $M_0$, $M_1$, $\gamma$, $M$, $\alpha$.
\end{lem}

\begin{proof} [Proof of Lemma ~\ref{lem:C^1,1}]
The above Lemma follows by Lemma \ref{lem:H^4} and by Sobolev
embedding theorem, \cite{l:ad}.
\end{proof}

\begin{rem}
   \label{rem_minimal}
We have stated the above regularity lemmas under the regularity assumptions
made for our main Theorem \ref{theo:Main}, but we point out that,
as can be seen {}from the proof given in Section \ref{sec: regularity}, for the
validity of Lemma \ref{lem:H^3} it suffices to assume $C^{2,1}$ regularity of $\partial D$ and $C^{0,1}$ regularity of
the coefficients of the elasticity tensor $\mathbb{C}$.
\end{rem}

\medskip
\noindent

We also need the following trace-type inequality.

\begin{lem} [Trace-type inequality]
\label{lem:trace} Let $\Omega$ be a bounded domain in $\R^{2}$,
satisfying \eqref{eq:bound_area} and \eqref{eq:reg_Omega}, and let
$D$ be a simply connected open subset of $\Omega$ satisfying
\eqref{eq:compactness} and \eqref{eq:reg_D}. Let the fourth order
tensor $\mathbb{P}$ be defined by \eqref{eq:P_def} and satisfying
\eqref{eq:sym-conditions-C-components}, \eqref{eq:3.bound_quantit}
and \eqref{eq:3.convex}. Let $\widehat M \in
H^{-\frac{1}{2}}(\partial \Omega,\R^2)$ satisfy
\eqref{eq:supp_M}--\eqref{eq:M_comp}, with $\Sigma$ satisfying
\eqref{eq:small_enough}. Let $w \in H^{2}(\Omega\setminus
\overline{D})$ be the unique weak solution of the problem
\eqref{eq:dir-pbm-incl-rig-1}--\eqref{eq:equil-rigid-incl}.  We
have
\begin{equation}
    \label{M_hat_above}
\|\widehat M\|_{H^{-\frac{1}{2}}(\partial \Omega,\R^2)}\leq
C\|\nabla ^2w\|_{L^2(\Omega\setminus \overline{D})},
\end{equation}
where $C>0$ only depends on $M_0$, $M_1$,
$\delta_0$ and $M$.
\end{lem}

\begin{rem}
   \label{rem:trace}
The proof of the above lemma can be obtained by adapting to the
mixed problem the proof of the analogous Lemma $7.1$ in
\cite{M-R-V1}, which was established for the solutions to a
Neumann problem. Indeed, the above result holds under weaker
assumptions: it suffices that the coefficients of the elasticity
tensor $\mathbb{C}$ are of class $L^\infty$ and that $\partial D$
is of class $C^{2,1}$.

\end{rem}

\noindent Next lemma ensures regularity of the boundaries of a
family of domains approximating the sets $D^\rho_i=\{x\in \Omega \
|\ \hbox{dist}(x, \partial D_i)<\rho\}$.

\begin{lem}[Regularized domains \cite{A-B-R-V}, Lemma 5.3]
  \label{lem:reg_domain}
Let $\Omega$ be a domain with boundary of class $C^{1,1}$ with
constants $\rho_0$, $M_0$ and satisfying \eqref{eq:bound_area}.
Let $D$ be a simply connected open subset of $\Omega$ with
boundary of class $C^{1,1}$ with constants $\rho_0$, $M_0$ and
satisfying \eqref{eq:compactness}. One can construct a family of
{\it regularized domains} $\tilde D ^h\subset\Omega$, for $0<h\leq
a \rho_0$, having boundary of class $C^1$ with constants
$\tilde\rho_0$ and $\tilde M_0$, such that
\begin{equation}
   \label{eq:reg_1}
   D \subset \tilde D^{h_1}\subset\tilde D^{h_2}, \qquad 0<h_1\leq h_2,
\end{equation}
\begin{equation}
   \label{eq:reg_2}
   \gamma_0 h\leq {\rm dist}(x,\partial D)\leq \gamma_1 h,\qquad
   {\rm for\ every\ } x\in\partial\tilde D^h,
\end{equation}
\begin{equation}
   \label{eq:reg_3}
   |\tilde D^h \setminus D|\leq \gamma_2 M_1 \rho_0h,
\end{equation}
\begin{equation}
   \label{eq:reg_4}
   |\partial\tilde  D^h|_{1}\leq \gamma_3 M_1
   \rho_0,
\end{equation}
for every $x\in\partial\tilde D^h$ there exists $y\in\partial D$
such that
\begin{equation}
   \label{eq:reg_5}
   |y-x|={\rm dist}(x,\partial D),\quad |\nu(x)-\nu(y)|\leq
   \gamma_4 \frac {h}{\rho_0},
\end{equation}
where $\nu(x)$, $\nu(y)$ denote the outer unit normals to $\tilde
D^h$ at $x$ and to $D$ at $y$ respectively, and $a$,
$\gamma_j$, $j=0, 1,...,4$, and the ratios
$\frac{\tilde\rho_0}{\rho_0}$ and $\frac{\tilde M_0}{M_0}$ only
depend on $M_0$. Here $|\cdot|_1$ denotes the
$1$- dimensional measure.
\end{lem}

\begin{proof} [Proof of Proposition \ref{prop:LPS}]
The proof of this Proposition is rather technical and follows the lines of the proof of
Proposition 4.2 in \cite{M-R-V5}, which refers to the Neumann problem for the plate equation
under the same general assumptions on the plate tensor and on the couple field. For this reason,
in order to simplify the present exposition, we illustrate the main ideas underlying the proof
of the first part of the result which is essentially the same as in the proof of
Proposition 4.2 in \cite{M-R-V5}.

The main idea is that of applying the three spheres inequality \eqref{eq:3sph} over a chain of disks centered
at points of a path connecting $\bar x$ to any point $y\in(\Omega\setminus \overline{D})_{s\rho}$, thus estimating $\int_{B_\rho(y)}|\nabla^2 w|^2$
in terms of $\int_{B_\rho(\bar x)}|\nabla^2 w|^2$. In doing this, in order to reduce the number of iterations, which is responsible of the deterioration of the estimate, we exploit the regularity of the boundary components $\partial\Omega$ and $\partial D$
of $\Omega\setminus \overline{D}$
(in fact Lipschitz regularity suffices for this step), constructing suitable cones near the boundary, inside which one can apply the three spheres inequality over a chain of disks tangent to the cones.

It is not restrictive to assume that $\rho_0=1$.

By a covering argument, one obtains the following estimate, which is the analogous of $(5.36)$ in \cite{M-R-V5}:
\begin{equation}
   \label{eq:reverse}
\int_{B_{\rho}(\overline{x})}|\nabla^2 w|^2 \geq
\int_{\Omega\setminus \overline{D}}|\nabla^2 w|^2
\left(\frac{C'\rho^2\int_{(\Omega\setminus \overline{D})_{(s+1)\rho}}|\nabla^2
w|^2}{\int_{\Omega\setminus \overline{D}}|\nabla^2
w|^2}\right)^{\delta_{\chi}^{-A_1-B_1\log\frac{1}{\rho}}}, \forall \rho\leq \bar\rho
\end{equation}
where $\delta_\chi$, $0<\delta_\chi<1$, $\bar\rho$, $B_1$ and $C'$ are positive constants only depending on $\gamma$, $M$, $\delta_1$ and $M_0$,
whereas $A_1>0$
only depends on $\gamma$, $M$, $\delta_1$, $M_0$ and $M_1$.

Let us set
\begin{equation}
   \label{eq:riscrittura}
   \frac{\int_{(\Omega\setminus \overline{D})_{(s+1)\rho}}|\nabla^2
w|^2}{\int_{\Omega\setminus \overline{D}}|\nabla^2
w|^2} = 1- \frac{\int_{(\Omega\setminus \overline{D})\setminus(\Omega\setminus \overline{D})_{(s+1)\rho}}|\nabla^2
w|^2}{\int_{\Omega\setminus \overline{D}}|\nabla^2
w|^2}.
\end{equation}

By H\"{o}lder inequality
\begin{equation}
  \label{eq:holder}
   \|\nabla^2 w\|^2_{L^2\left({(\Omega\setminus \overline{D})\setminus(\Omega\setminus \overline{D})_{(s+1)\rho}}\right)}\leq
   \left|{(\Omega\setminus \overline{D})\setminus(\Omega\setminus \overline{D})_{(s+1)\rho}}\right|^{\frac{1}{2}}
   \|\nabla^2 w\|^2_{L^4\left({(\Omega\setminus \overline{D})\setminus(\Omega\setminus \overline{D})_{(s+1)\rho}}\right)},
\end{equation}
and by Sobolev inequality \cite{l:ad}
\begin{equation}
  \label{eq:sobolev}
   \|\nabla^2 w\|^2_{L^4\left(\Omega\setminus \overline{D}\right)}\leq C \|\nabla^2 w\|^2_{H^{\frac{1}{2}}(\Omega\setminus \overline{D})},
\end{equation}
we have
\begin{equation}
  \label{eq:hol-sob}
   \|\nabla^2 w\|^2_{L^2({(\Omega\setminus \overline{D})\setminus(\Omega\setminus \overline{D})_{(s+1)\rho}})}\leq
   C \left|{(\Omega\setminus \overline{D})\setminus(\Omega\setminus \overline{D})_{(s+1)\rho}}\right|^{\frac{1}{2}}
   \| w\|^2_{H^{\frac{5}{2}}(\Omega\setminus \overline{D})},
\end{equation}
where $C$ only depends on $M_0$, $M_1$.

By interpolating the estimates \eqref{eq:w_stima_H2} and \eqref{eq:reg_H3}, we have
\begin{equation}
  \label{eq:5/2_2}
   \| w\|_{H^{\frac{5}{2}}(\Omega\setminus \overline{D})}\leq C\|\widehat{M}\|_{L^2(\partial\Omega,\R^2)},
\end{equation}
where $C$ only depends on $M_0$, $M_1$, $\gamma$ and $M$.

Moreover
\begin{equation}
  \label{eq:AR}
   \left|(\Omega\setminus \overline{D})\setminus(\Omega\setminus \overline{D})_{(s+1)\rho}\right|\leq C\rho,
\end{equation}
with $C$ only depending on $M_0$ and $M_1$, see for details (A.3)
in \cite{Al-Ro}. {}From \eqref{eq:hol-sob}, \eqref{eq:5/2_2} and
\eqref{eq:AR} we have
\begin{equation}
  \label{eq:stima_num}
   \int_{(\Omega\setminus \overline{D})\setminus(\Omega\setminus \overline{D})_{(s+1)\rho}}|\nabla^2 w|^2\leq
   C\rho^{\frac{1}{2}}\|\widehat{M}\|^2_{L^2(\partial\Omega,\R^2)},
\end{equation}
where $C$ only depends on $M_0$, $M_1$, $\gamma$, $M$.

By \eqref{eq:riscrittura}, \eqref{eq:stima_num}
and \eqref{M_hat_above} we obtain that there exists $\widetilde{\rho}>0$, only depending on $M_0$, $M_1$,
$\delta_0$, $\gamma$, $M$ and $F$, such that for every
$\rho\leq\widetilde{\rho}$ we have
\begin{equation}
  \label{eq:Frequency}
   \frac{\int_{(\Omega\setminus \overline{D})_{(s+1)\rho}}|\nabla^2 w|^2}{\int_{\Omega\setminus \overline{D}}|\nabla^2 w|^2}\geq \frac{1}{2}.
\end{equation}
By \eqref{eq:reverse}, \eqref{eq:Frequency} and \eqref{M_hat_above}, we have that for $\rho\leq \min_\{\bar\rho, \widetilde{\rho}\}$,
\begin{equation}
   \label{eq:reverse_last}\int_{B_{\rho}(\overline{x})}|\nabla^2 w|^2
   \geq \left(\widetilde{C}\rho^2\right)^{\delta_{\chi}^{-A_1-B_1\log\frac{1}{\rho}}} \|\widehat{M}\|_{H^{-\frac{1}{2}}(\partial \Omega,\R^2)}^2,
\end{equation}
where $\tilde C>0$ only depends on $\gamma$, $M$, $\delta_1$, $M_0$,
$M_1$ and $\delta_0$. Let us take $\rho\leq\widetilde{C}$. Noticing
that $|\log \rho|\leq \frac{1}{\rho}$, for every $\rho>0$, and that
$\tilde\rho<1$, by straightforward computations we obtain that
\eqref{eq:LPS} holds with $A=3\exp(A_1|\log\delta_\chi|)$,
$B=|\log\delta_\chi|B_1+1$ for every $\rho\leq\rho^*$ with
$\rho^*=\min\{\bar\rho, \frac{\tilde\rho}{s+1},\tilde C\}$, $\rho^*$ only
depending on $\gamma$, $M$, $\delta_1$, $M_0$, $M_1$, $\delta_0$, and
$F$.

If $\rho>\rho^*$ and $\bar
x\in(\Omega\setminus \overline{D})_{s\rho}\subset(\Omega\setminus \overline{D})_{s\rho^*}$, then, by applying \eqref{eq:LPS}
with $\rho=\rho^*$,  we have
\begin{equation}
   \label{eq:rho>}
   \int_{B_{\rho}(\bar x)}|\nabla^2 w|^2
   \geq \int_{B_{\rho^*}(\bar x)}|\nabla^2 w|^2
   \geq C^*\|\widehat{M}\|_{H^{-\frac{1}{2}}(\partial \Omega,\R^2)}^2,
\end{equation}
where $C^*$ only depends on  $\gamma$, $M$, $\delta_1$, $M_0$, $M_1$, $\delta_0$ and
$F$.

Since $\bar x\in (\Omega\setminus \overline{D})_{s\rho}$, we have that
\begin{equation}
\label{diam_lower}
 \hbox{diam}(\Omega\setminus \overline{D})\geq 2s\rho,
\end{equation}
and, on the other hand,
\begin{equation}
\label{diam_upper}
 \hbox{diam}(\Omega\setminus \overline{D})\leq C_2,
\end{equation}
with $C_2$ only depending on $M_0$ and $M_1$, so that
\begin{equation}
   \label{eq:upper_for_rho>}
   \frac{2s}{C_2}\leq\frac{1}{\rho}.
\end{equation}
By \eqref{eq:rho>} and \eqref{eq:upper_for_rho>}, we have
\begin{equation}
   \label{eq:rho>finale}
   \int_{B_{\rho}(\bar x)}|\nabla^2 w|^2
   \geq \frac{C}{\exp\left[A\left(\frac{1}{\rho}\right)^B\right]}
   \|\widehat{M}\|_{H^{-\frac{1}{2}}(\partial \Omega,\R^2)}^2,
\end{equation}
with $C=C^*\exp\left[A\left(\frac{2s}{C_2}\right)^B\right]$.

\end{proof}

\begin{proof} [Proof of Proposition \ref{prop:Cauchy1}]

Let $G$ be the connected component of $\Omega\setminus \left(\overline{D_1\cup D_2}\right)$ such that
$\Sigma\subset \partial G$. Let us prove \eqref{eq:Cauchy1}--\eqref{eq:omega_cauchy_loglog}, the proof of
\eqref{eq:Cauchy1bis}--\eqref{eq:omega_cauchy_loglog} being analogous.

Let $\overline{g}\in \mathcal{A}$ the affine function realizing the minimum in \eqref{eq:small_L2} and let us set

\begin{equation}
    \label{eq:def-w}
    w = w_1 - w_2 - \overline{g}.
\end{equation}

\noindent
\textbf{Step 1.} Let $\overline{g}\equiv 0$.

In this case $w = w_1 - w_2$.

It is not restrictive to assume $\epsilon<\rho_0^2\|\widehat{M}\|_{H^{-\frac{1}{2}}(\partial \Omega,\R^2)}\widetilde{\mu}$,
where $\widetilde{\mu}$, $0<\widetilde{\mu}<e^{-1}$ is a constant only depending on $\gamma$, $M$, $\delta_1$,
$M_0$, $M_1$, which will be chosen later on.

Let $\vartheta=\min\left\{a, \frac{3}{4\gamma_1}, \frac{h_0}{4\gamma_0\sqrt{1+M_0^2}}\right\}$, where
$h_0$, $0<h_0<1$, only depending on $M_0$, is
such that $\Omega_h$ is connected for every $h\leq h_0$ (see Prop. 5.5
in \cite{A-R-R-V}), and where $a$, $\gamma_0$, $\gamma_1$ have been introduced in Lemma \ref{lem:reg_domain}.
Let $\overline{\rho}=\vartheta\rho_0$ and let $\rho\leq\overline{\rho}$.

Let us denote by $\widetilde{V}_\rho$ the connected component of $\overline{\Omega}\setminus (\widetilde{D}_1^\rho\cup \widetilde{D}_2^\rho)$ which
contains $\partial\Omega$, where $\widetilde{D}_i^\rho$ are the regularized domains introduced in Lemma \ref{lem:reg_domain}. We have

\begin{equation}
   \label{eq:allargamento}
D_2 \setminus \overline{D_1} \subset \left ( (\tilde D_1^\rho \setminus
\overline{D_1} ) \setminus \overline{G} \right ) \cup \left (
(\Omega \setminus \tilde V_\rho ) \setminus \overline{\tilde
D_1^\rho} \right ),
\end{equation}
\begin{equation}
\label{eq:bordo_allargamento}
\partial \left (
(\Omega \setminus \tilde V_\rho ) \setminus \overline{\tilde
D_1^\rho} \right ) = \tilde\Gamma_{1}^\rho
\cup\tilde\Gamma_{2}^\rho,
\end{equation}
where $\tilde\Gamma_{2}^\rho=\partial\tilde D_{2}^\rho\cap\partial
\tilde V_\rho$ and $\tilde\Gamma_{1}^\rho\subset\partial\tilde
D_{1}^\rho$.

We have
\begin{equation}
  \label{eq:stab1}
  \int_{D_2 \setminus \overline{D_1}} |\nabla^2 w_1|^2\leq
  \int_{(\tilde D_1^\rho \setminus
\overline{D_1} ) \setminus \overline{G}}|\nabla^2 w_1|^2+ \int_{(\Omega \setminus \tilde V_\rho ) \setminus \overline{\tilde
D_1^\rho}}|\nabla^2 w_1|^2.
\end{equation}
By \eqref{eq:reg_classica} and \eqref{eq:reg_3} we have
\begin{equation}
   \label{eq:stab2}
   \int_{(\tilde D_1^\rho \setminus
\overline{D_1} ) \setminus \overline{G}}|\nabla^2 w_1|^2
   \leq
   C\rho_0^2 \|\widehat{M}\|_{H^{-\frac{1}{2}}(\partial \Omega, \R^2)}^2 \frac{\rho}{\rho_0} ,
\end{equation}
with $C>0$ only depending on $\gamma$, $M$, $M_0$ and $M_1$.

By applying the divergence theorem and by \eqref{eq:3.convex} we have
\begin{multline}
   \label{eq:byparts}
   \int_{(\Omega \setminus \tilde V_\rho ) \setminus \overline{
\tilde D_1^\rho}}|\nabla^2 w_1|^2\leq\gamma^{-1} \int_{\tilde\Gamma_{1}^\rho}B_2(w_1)w_{1,\nu}+\widetilde{B}_2(w_1)w_{1,s}+B_3(w_1)w_1+\\
+\gamma^{-1}\int_{\tilde\Gamma_{2}^\rho}B_2(w_1)w_{1,\nu}+\widetilde{B}_2(w_1)w_{1,s}+B_3(w_1)w_1,
\end{multline}

where
\begin{equation}
   \label{eq:B_2}
   B_2(w_1)=(\mathbb{P} \nabla^2w_1)\nu\cdot\nu,
\end{equation}
\begin{equation}
   \label{eq:B'_2}
   \widetilde{B}_2(w_1)=(\mathbb{P} \nabla^2w_1)\nu\cdot\tau,
\end{equation}
\begin{equation}
   \label{eq:B_3}
   B_3(w_1)=-\divrg(\mathbb{P} \nabla^2w_1)\cdot\nu,
\end{equation}
with $\nu$ denoting the  unit outer normal to $(\Omega \setminus \tilde V_\rho ) \setminus \overline{
\tilde D_1^\rho}$, and $s$ an arclength with associated parametrization
$\varphi(s)$,  such that
$\varphi'(s)=\tau(\varphi(s))$, where $\tau=e_3 \times \nu$.

Let $x\in\tilde\Gamma_{1}^\rho$. By \eqref{eq:reg_2},
dist$(x,\partial D_1)\leq\gamma_1 \rho$. Let $y\in
\partial D_1$ such that $|y-x|=
\hbox{dist}(x,\partial D_1)\leq\gamma_1 \rho$. Since $\gamma_1 \rho\leq \frac{3}{4}\rho_0$ and by \eqref{eq:compactness}
we have that $x\in (\Omega\setminus \overline{D_1})\setminus \mathcal{U}^{\rho_0/8}$ and \eqref{eq:reg_classica} applies.
Recalling also that $w_1\equiv 0$, $\nabla w_1\equiv 0$ on $\partial D_1$, we have that
\begin{equation}
   \label{eq:w_1small}
   |w_1 (x)|=|w_1 (x)-w_1 (y)|
   \leq
   C\rho_0^2
   \|\widehat{M}\|_{H^{-\frac{1}{2}}(\partial \Omega, \R^2)}
   \frac{\rho}{\rho_0},
\end{equation}

\begin{equation}
   \label{eq:w_1,nu_small}
   |w_{1,\nu} (x)|=|(\nabla w_1 (x)-\nabla w_1 (y))\cdot\nu(x)|
   \leq
   C\rho_0
   \|\widehat{M}\|_{H^{-\frac{1}{2}}(\partial \Omega, \R^2)}
   \frac{\rho}{\rho_0},
\end{equation}

\begin{equation}
   \label{eq:w_1,s_small}
   |w_{1,s} (x)|=|(\nabla w_1 (x)-\nabla w_1 (y))\cdot\tau(x)|
   \leq
   C\rho_0
   \|\widehat{M}\|_{H^{-\frac{1}{2}}(\partial \Omega, \R^2)}
   \frac{\rho}{\rho_0},
\end{equation}
where $C$ only depends on $\gamma$, $M$, $M_0$ and $M_1$.

{From} \eqref{eq:reg_classica}, \eqref{eq:w_1,nu_small}, \eqref{eq:w_1,s_small} and \eqref{eq:reg_4} we have

\begin{equation}
   \label{eq:stab3}
   \left|\int_{\tilde \Gamma_1^\rho}B_2(w_1)w_{1,\nu}+\widetilde{B}_2(w_1)w_{1,s}\right|
   \leq
   C\rho_0^2 \|\widehat{M}\|_{H^{-\frac{1}{2}}(\partial \Omega, \R^2)}^2 \frac{\rho}{\rho_0},
\end{equation}
where $C$ only depends on $\gamma$, $M$, $M_0$ and $M_1$.

Moreover, by \eqref{eq:w_1small} and \eqref{eq:reg_H4}, we have
\begin{multline}
   \label{eq:stab4}
   \left|\int_{\tilde \Gamma_1^\rho}B_3(w_1)w_1\right|\leq
   C\rho_0\|B_3(w_1)\|_{L^2(\tilde \Gamma_1^\rho)}\|w_1\|_{L^2(\tilde\Gamma_1^\rho)}\leq\\
   \leq\frac{C}{\rho_0^2}\|w_1\|_{H^3(\tilde\Gamma_1^\rho)}\|w_1\|_{L^\infty(\tilde\Gamma_1^\rho)}\leq
   \frac{C}{\rho_0^2}\|w_1\|_{H^4\left((\Omega\setminus \overline{D_1})\setminus \mathcal{U}^{\rho_0/8}\right)}
   \|w_1\|_{L^\infty(\tilde\Gamma_1^\rho)}\leq
   \\
   \leq
   C\rho_0^2 \|\widehat{M}\|_{H^{-\frac{1}{2}}(\partial \Omega, \R^2)}^2 \frac{\rho}{\rho_0},
\end{multline}
where $C$ only depends on $\gamma$, $M$, $M_0$ and $M_1$.

Let $x\in\tilde\Gamma_{2}^\rho$. By \eqref{eq:reg_2},
dist$(x,\partial D_2)\leq\gamma_1 \rho$. Let $y\in
\partial D_2$ such that $|y-x|=
\hbox{dist}(x,\partial D_2)\leq\gamma_1 \rho$.
Since $\gamma_1 \rho\leq \frac{3}{4}\rho_0$
we have that $x\in (\Omega\setminus \overline{D_2})\setminus \mathcal{U}^{\rho_0/8}$ and \eqref{eq:reg_classica} applies.
Recalling also that $w_2\equiv 0$, $\nabla w_2\equiv 0$ on $\partial D_2$, we have that
\begin{equation}
   \label{eq:w_1smallbis}
   |w_1 (x)|=|w (x)+w_2 (x)-w_2 (y)|
   \leq |w (x)|+
   C\rho_0^2
   \|\widehat{M}\|_{H^{-\frac{1}{2}}(\partial \Omega, \R^2)}
   \frac{\rho}{\rho_0},
\end{equation}

\begin{equation}
   \label{eq:w_1,nu_smallbis}
   |w_{1,\nu} (x)|=|w,_\nu (x)+\left(\nabla w_2 (x)-\nabla w_2 (y)\right)\cdot\nu(x)|
   \leq |\nabla w (x)|+
   C\rho_0
   \|\widehat{M}\|_{H^{-\frac{1}{2}}(\partial \Omega, \R^2)}
   \frac{\rho}{\rho_0},
\end{equation}

\begin{equation}
   \label{eq:w_1,s_smallbis}
   |w_{1,s} (x)|=|w,_s (x)+\left(\nabla w_2 (x)-\nabla w_2 (y)\right)\cdot\tau(x)|
   \leq |\nabla w (x)|+
   C\rho_0
   \|\widehat{M}\|_{H^{-\frac{1}{2}}(\partial \Omega, \R^2)}
   \frac{\rho}{\rho_0},
\end{equation}
where $C$ only depends on $\gamma$, $M$, $M_0$ and $M_1$.

{From} \eqref{eq:reg_classica} \eqref{eq:w_1,nu_smallbis}, \eqref{eq:w_1,s_smallbis} and \eqref{eq:reg_4} we have

\begin{equation}
   \label{eq:stab5}
   \left|\int_{\tilde \Gamma_2^\rho}B_2(w_1)w_{1,\nu}+\widetilde{B}_2(w_1)w_{1,s}\right|
   \leq
   C\rho_0^2 \|\widehat{M}\|_{H^{-\frac{1}{2}}(\partial \Omega, \R^2)}^2 \frac{\rho}{\rho_0}
   +C\rho_0 \|\widehat{M}\|_{H^{-\frac{1}{2}}(\partial \Omega, \R^2)}\max_{\tilde \Gamma_2^\rho}|\nabla w|,
\end{equation}
where $C$ only depends on $\gamma$, $M$, $M_0$ and $M_1$.

Moreover, by \eqref{eq:w_1smallbis} and \eqref{eq:reg_H4}, we have, similarly to above,
\begin{equation}
   \label{eq:stab6}
   \left|\int_{\tilde \Gamma_2^\rho}B_3(w_1)w_1\right|
   \leq
   C\rho_0^2 \|\widehat{M}\|_{H^{-\frac{1}{2}}(\partial \Omega, \R^2)}^2 \frac{\rho}{\rho_0}+
   C \|\widehat{M}\|_{H^{-\frac{1}{2}}(\partial \Omega, \R^2)}\max_{\tilde \Gamma_2^\rho}| w|,
\end{equation}
where $C$ only depends on $\gamma$, $M$, $M_0$ and $M_1$.

By \eqref{eq:stab1}--\eqref{eq:byparts}, \eqref{eq:stab3}, \eqref{eq:stab4}, \eqref{eq:stab5}
and \eqref{eq:stab6} we have
\begin{equation}
  \label{eq:stab7}
  \int_{D_2 \setminus \overline{D_1}} |\nabla^2 w_1|^2\leq
  C\rho_0^2 \|\widehat{M}\|_{H^{-\frac{1}{2}}(\partial \Omega, \R^2)}^2 \frac{\rho}{\rho_0}
  +C \|\widehat{M}\|_{H^{-\frac{1}{2}}(\partial \Omega, \R^2)}
  \|w\|_{C^1(\partial \tilde V^\rho\setminus \partial \Omega)},
\end{equation}
where $\|w\|_{C^1(\partial \tilde V^\rho\setminus \partial \Omega)}=
\max_{\partial \tilde V^\rho\setminus \partial \Omega}|w|+\rho_0\max_{\partial \tilde V^\rho\setminus \partial \Omega}|\nabla w|$, and
$C$ only depends on $\gamma$, $M$, $M_0$ and $M_1$.

Let us notice that $w\in H^2(G)$ satisfies

\begin{center}
\( {\displaystyle \left\{
\begin{array}{lr}
     {\rm div}({\rm div} (
      {\mathbb P}\nabla^2 w))=0,
      & \mathrm{in}\ G,
        \vspace{0.25em}\\
      ({\mathbb P} \nabla^2 w)n\cdot n=0, & \mathrm{on}\ \Sigma,
          \vspace{0.25em}\\
      {\rm div}({\mathbb P} \nabla^2 w)\cdot n+(({\mathbb P} \nabla^2
      w)n\cdot \tau),_s
      =0, & \mathrm{on}\ \Sigma,
        \vspace{0.25em}\\
      \|w\|_{L^2(\Sigma)} +\rho_0\|w,_n\|_{L^2(\Sigma)}\leq \epsilon.
          \vspace{0.25em}\\
\end{array}
\right. } \) \vskip -7.3em
\begin{eqnarray}
& & \label{eq:smallCauchy1}\\
& & \label{eq:smallCauchy2}\\
& & \label{eq:smallCauchy3}\\
& & \label{eq:smallCauchy4}
\end{eqnarray}

\end{center}
By the regularity of $w$ near $\Sigma$ (see Theorem \ref{theo:BoundNeuNonhom}) and by standard interpolation inequalities, we can apply to $w$ the stability estimate for the Cauchy
problem contained in \cite[Theorem 3.8]{M-R-V4}. Therefore, noticing that, by \eqref{eq:w_stima_H2},
$\|w\|_{L^2(G)}\leq \|w_1\|_{L^2(\Omega\setminus \overline{D_1})} + \|w_2\|_{L^2(\Omega\setminus \overline{D_2})}
\leq C\rho_0^2 \|\widehat{M}\|_{H^{-\frac{1}{2}}(\partial \Omega, \R^2)}$, with $C$ only depending on $\gamma$, $M_0$, $M_1$,
we have
\begin{equation}
  \label{eq:stab8}
  \|w\|_{L^2(R_{\frac{\rho_0}{2M_0},\frac{\rho_0}{2}}(P_0)\cap\Omega)}\leq C\epsilon^\tau
  \rho_0^{2(1-\tau)}\|\widehat{M}\|_{H^{-\frac{1}{2}}(\partial \Omega, \R^2)}^{1-\tau},
\end{equation}
where $C>0$, $\tau\in (0,1)$ only depend on $\gamma$, $M$, $\delta_1$, $M_0$ and $M_1$.

Let
\begin{equation}
  \label{eq:z_0}
  z_0=P_0-\frac{\rho_0}{4}n,
\end{equation}
where $n$ denotes as usual the outer unit normal to $\partial\Omega$ at $P_0$, and notice that
$B_{\frac{\rho_0}{4\sqrt{1+M_0^2}}}(z_0)\subset R_{\frac{\rho_0}{2M_0},\frac{\rho_0}{2}}(P_0)\cap\Omega$.

By \eqref{eq:stab8},

\begin{equation}
  \label{eq:stab9}
  \|w\|_{L^2\left(B_{\frac{\rho_0}{4\sqrt{1+M_0^2}}}(z_0)\right)}\leq C\epsilon^\tau
  \rho_0^{2(1-\tau)}\|\widehat{M}\|_{H^{-\frac{1}{2}}(\partial \Omega, \R^2)}^{1-\tau}.
\end{equation}
Let us consider the set $\tilde V^\rho\cap
\Omega_{\frac{\rho_0h_0}{4\sqrt{1+M_0^2}}}$. By the choice of
$\bar\rho$ and recalling that $h_0<1$, we have by straightforward computations that this set is
connected and contains $z_0$. If $x\in \tilde V^\rho\cap
\Omega_{\frac{\rho_0h_0}{4\sqrt{1+M_0^2}}}$ then, by the choice of
$\bar\rho$, $\hbox{dist}(x,\partial\Omega)>
\frac{\rho_0h_0}{4\sqrt{1+M_0^2}}\geq \gamma_0\rho$ and by
\eqref{eq:reg_2} $\hbox{dist}(x,\partial D_1\cup \partial D_2)\geq
\gamma_0\rho$. Therefore any disk of center in $\tilde V^\rho\cap
\Omega_{\frac{\rho_0h_0}{4\sqrt{1+M_0^2}}}$ and radius
$\gamma_0\rho$ is contained in $G$.

Let $x$ be any point in $\partial \tilde
V_\rho\setminus\partial\Omega$ and let $\gamma$ be a path in
$\tilde V^\rho\cap \Omega_{\frac{\rho_0h_0}{4\sqrt{1+M_0^2}}}$
joining $x$ to $z_0$. Let us define $\{x_i\}$, $i=1,...,s$, as
follows: $x_1=z_0$, $x_{i+1}=\gamma(t_i)$, where $t_i=\max\{t
\hbox{  s. t. }|\gamma(t)-x_i|=\frac{\gamma_0 \rho }{2} \}$, if
$|x_i-x|>\frac{\gamma_0 \rho }{2}$, otherwise let $i=s$ and stop
the process. By construction, the disks $B_{\frac{\gamma_0 \rho
}{4}}(x_i)$ are pairwise disjoint, $|x_{i+1}-x_i|=\frac{\gamma_0
\rho }{2}$, for $i=1,..., s-1$, $|x_s-x|\leq \frac{\gamma_0 \rho
}{2}$. Hence we have $s\leq S\left(\frac{\rho_0}{\rho}\right)^2$,
with $S=\frac{16M_1}{\pi^2\gamma_0^2}$ only depending on $M_0$
and $M_1$.

An iterated application of the three spheres inequality
\eqref{eq:3sph} to $w$ over the disks of center $x_i$ and radii $r_1=\frac{\gamma_0 \rho }{4}$,
$r_2=\frac{3\gamma_0 \rho }{4}$, $r_3=\gamma_0 \rho $, gives that
for every $\rho$, $0<\rho\leq \bar \rho$,
\begin{equation}
  \label{eq:stab10}
  \int_{B_{\frac{\gamma_0 \rho }{4}}(x)}|w|^2
  \leq
  C
  \left(
  \int_G|w|^2
  \right)^{1-\delta^s}
  \left(
  \int_{B_{\frac{\gamma_0 \rho }{4}}(z_0)}
  |w|^2
  \right)^{\delta^s},
\end{equation}
where $\delta$, $0<\delta<1$, $C\geq 1$, only depend on $\gamma$,
$M$ and $\delta_1$.

Since $B_{\frac{\gamma_0 \rho }{4}}(z_0)\subset B_{\frac{ \rho_0
}{4\sqrt{1+M_0^2}}}(z_0)$, by \eqref{eq:stab9}, \eqref{eq:stab10}
and \eqref{eq:w_stima_H2} we have
\begin{equation}
  \label{eq:stab11}
  \int_{B_{\frac{\gamma_0 \rho }{4}}(x)}|w|^2
  \leq
  C \rho_0^6
  \|\widehat{M}\|_{H^{- \frac{1}{2} }(\partial \Omega, \R^2)}^2
  \tilde\epsilon ^{2\tau\delta^s},
\end{equation}
where
\begin{equation}
     \label{eq:epsilon_tilde}
    \tilde\epsilon = \frac {\epsilon}
    { \rho_0^{2}
    \|\widehat{M}\|_{H^{- \frac{1}{2} }(\partial \Omega, \R^n2)} }.
\end{equation}

Let us recall the following interpolation inequality
\begin{equation}
  \label{eq:interpolation}
  \|v\|_{L^\infty (B_t)}
  \leq
  C\left(
  \left(
  \int_{B_t}|v|^2
  \right) ^{ \frac{1}{4} }
  |\nabla v|_{L^\infty (B_t)}^{\frac{1}{2}}
  +
  \frac{1}{t}
  \left(
  \int_{B_t}|v|^2
  \right)^{\frac{1}{2}}
  \right),
\end{equation}
which holds for any function $v\in W^{1,\infty} (B_t)$.

By applying \eqref{eq:interpolation} to $w$ in $B_{\frac{\gamma_0
\rho }{4}}(x)$ and by \eqref{eq:stab11} and
\eqref{eq:reg_classica}, we have
\begin{equation}
  \label{eq:stab12}
  \|w\|_{L^{\infty}(B_{\frac{\gamma_0 \rho }{4}}(x))}
  \leq
  C\rho_0^2 \|\widehat{M}\|_{H^{- \frac{1}{2} }(\partial \Omega, \R^2)}
  \frac{\rho_0}{\rho}
   \tilde\epsilon^{\frac{\tau\delta^s}{2} }.
\end{equation}
By the following interpolation inequality (see for instance
\cite{G-T})
\begin{equation}
  \label{eq:interpolation2}
  \|v\|_{C^1(B_t)}
  \leq
  C\|v\|_{C^0(B_t)}^{\frac{1}{2}}\|v\|_{C^2(B_t)}^{\frac{1}{2}}.
\end{equation}
we have, recalling \eqref{eq:reg_classica}, that
\begin{equation}
  \label{eq:stab13}
  \|w\|_{C^1(B_{\frac{\gamma_0 \rho }{4}}(x))}
  \leq
  C\rho_0^2 \|\widehat{M}\|_{H^{- \frac{1}{2} }(\partial \Omega, \R^2)}
  \left(\frac{\rho_0}{\rho}\right)^{\frac{1}{2}}
   \tilde\epsilon^{\frac{\tau\delta^s}{4} },
\end{equation}
so that
\begin{equation}
  \label{eq:stab14}
   \|w\|_{C^1(\partial \tilde V^\rho\setminus \partial \Omega)}\leq C\rho_0^2 \|\widehat{M}\|_{H^{-
\frac{1}{2} }(\partial \Omega, \R^2)}
  \left(\frac{\rho_0}{\rho}\right)^{\frac{1}{2}}
   \tilde\epsilon^{\frac{\tau\delta^s}{4} }.
\end{equation}
{}From \eqref{eq:stab7} and \eqref{eq:stab14} we have that, for
every $\rho\leq\bar\rho$,
\begin{equation}
  \label{eq:stab15}
  \int_{D_2 \setminus \overline{D_1}} |\nabla^2 w_1|^2\leq
  C\rho_0^2 \|\widehat{M}\|_{H^{-\frac{1}{2}}(\partial \Omega, \R^2)}^2 \left(\frac{\rho}{\rho_0}
  +\frac{\rho_0}{\rho} \tilde\epsilon^{\frac{\tau\delta^s}{4}
  }\right).
\end{equation}
Let us set
$\bar\mu=\exp\left\{-\frac{4}{\tau}\exp\left(\frac{2S|\log\delta|}{\vartheta^2}\right)\right\}$,
$\tilde\mu=\min\left\{\bar\mu,\exp\left(-\frac{16}{\tau^2}\right)\right\}$.
We have that $\tilde\mu$, $\tilde\mu<e^{-1}$, only depends on
$\gamma$, $M$, $\delta_1$, $M_0$ and $M_1$. Let $\tilde
\epsilon\leq\tilde\mu$, and let
\begin{equation}
    \label{rho_di_eps}
    \rho(\tilde\epsilon)=\rho_0
    \left(
    \frac{2S|\log\delta|}{\log|\log \tilde\epsilon
    ^{\frac{\tau}{4}}|}
    \right)^{\frac{1}{2}}.
\end{equation}
Since $\rho(\tilde\epsilon)$ is increasing in $(0,e^{-1})$ and
since $\rho(\tilde\mu)\leq \rho(\bar\mu)=\rho_0\theta= \bar \rho$,
we can apply inequality \eqref{eq:stab15} with
$\rho=\rho(\tilde\epsilon)$, obtaining
\begin{equation}
  \label{eq:stab16}
  \int_{D_2 \setminus \overline{D_1}}| {\nabla}^2 w_1|^2
  \leq
  C \rho_0^2
  \|\widehat{M}\|_{H^{- \frac{1}{2}  }(\partial \Omega, \R^2)}^2
  \left(
  \log\left|\log\tilde\epsilon^{\frac{\tau}{4}}\right|
  \right)^{-\frac{1}{2}},
\end{equation}
where $C$ only depends on $\gamma$, $M$, $\delta_1$, $M_0$ and
$M_1$.

Since $\tilde\epsilon\leq \exp(-\frac{16}{\tau^2})$, we have that
$\log\frac{\tau}{4}\geq-\frac{1}{2}\log|\log\tilde\epsilon|$, so
that
\begin{equation}
  \label{eq:stab17}
  \log\left|\log\tilde\epsilon ^{\frac{\tau}{4}}\right|
  \geq
  \frac{1}{2}
  \log\left| \log\tilde\epsilon\right|.
\end{equation}
{}From \eqref{eq:stab16} and \eqref{eq:stab17} we have
\begin{equation}
  \label{eq:stab16bis}
  \int_{D_2 \setminus \overline{D_1}} |{\nabla}^2 w_1|^2
  \leq
  \rho_0^2 \|\widehat{M} \|_{H^{- \frac{1}{2} }(\partial \Omega,
  \R^2)}^2 \omega(\tilde \epsilon),
\end{equation}
with
\begin{equation}
  \label{eq:omega}
  \omega(t)=C \left (\log |\log t| \right )^{ -\frac{1}{2}} \quad \hbox{for every \ } t<e^{-1},
\end{equation}
where $C>0$ is a constant only depending on $\gamma$, $M$,
$\delta_1$, $M_0$ and $M_1$.

\noindent \textbf{Step 2.} Let $\overline{g}\not \equiv 0$.

Let
\begin{equation}
    \label{eq:rho_tilde}
    \tilde \rho = \rho (\tilde \epsilon),
\end{equation}
where $\rho(\tilde \epsilon)$ is given by \eqref{rho_di_eps}.

We have two cases:

\begin{enumerate}[I)]
  \item $\partial\tilde D_1^{\tilde \rho} \cap \tilde \Gamma_2^{\tilde
  \rho}= \emptyset$;
  \item $\partial\tilde D_1^{\tilde \rho} \cap \tilde \Gamma_2^{\tilde
  \rho}\neq \emptyset$.
\end{enumerate}

When I) holds, there are three possible subcases:
\item {Ia)}
$\tilde D_1^{\tilde \rho} \cap \tilde D_2^{\tilde \rho} =
\emptyset$,
\item {Ib)} $\tilde D_1^{\tilde \rho} \subset \tilde
D_2^{\tilde \rho}$,
\item {Ic)} $\tilde D_2^{\tilde \rho} \subset
\tilde D_1^{\tilde \rho}$.

In case Ia) we have that $( \Omega \setminus \tilde V_{\tilde
\rho} ) \setminus \tilde D_1^{\tilde \rho} = \tilde D_2^{\tilde
\rho}$ and, therefore, $ \partial \left ( ( \Omega \setminus
\tilde V_{\tilde \rho} ) \setminus \tilde D_1^{\tilde \rho} \right
) =
\partial \tilde D_2^{\tilde \rho}$, whereas in case Ib)
$( \Omega \setminus \tilde V_{\tilde \rho} ) \setminus \tilde
D_1^{\tilde \rho} = \tilde D_2^{\tilde \rho}\setminus \tilde
D_1^{\tilde \rho}$.

For both cases, by applying the divergence theorem to $w_1$ in
$\tilde D_2^{\tilde \rho}$ and in $\tilde D_2^{\tilde
\rho}\setminus D_1$ respectively and by taking into account
\eqref{eq:equil-rigid-incl}, we have
\begin{equation}
    \label{eq:no-flux-u1-D2rho}
    \int_{\partial \tilde D_2^{\tilde \rho}}
    \left ( {\rm div}({\mathbb P} \nabla^2 w_1)\cdot \nu+(({\mathbb P} \nabla^2
  w_1)\nu\cdot \tau),_s \right )g - (({\mathbb P} \nabla^2 w_1)\nu\cdot \nu)
  g_{,\nu} =0, \\   \quad \hbox{for every } g\in \mathcal{A},
\end{equation}
with $\nu$ denoting the  unit outer normal to $D_2^{\tilde \rho}$, and $s$ an arclength with associated parametrization
$\varphi(s)$,  such that
$\varphi'(s)=\tau(\varphi(s))$, where $\tau=e_3 \times \nu$.

By applying the estimates of continuation from Cauchy data
\eqref{eq:stab14} obtained in the above step to the function $w$ defined by \eqref{eq:def-w}, we have
\begin{equation}
  \label{eq:stab18}
   \|w\|_{C^1(\partial \tilde V^\rho\setminus \partial \Omega)}\leq
   C\rho_0^2 \|\widehat{M}\|_{H^{-\frac{1}{2} }(\partial \Omega, \R^2)}
  \left(\frac{\rho_0}{\rho}\right)^{\frac{1}{2}}
   \tilde\epsilon^{\frac{\tau\delta^s}{4} }.
\end{equation}
By recalling that $w_i =0$, $\nabla w_i =0$ on $\partial D_i$,
$i=1,2$, and by arguing similarly to Step 1,
we have, for both cases, that
\begin{multline}
  \label{eq:int-D2menoD1}
  \int_{D_2 \setminus \overline{D_1}} |{\nabla}^2 w_1|^2
  \leq
  \int_{\tilde D_2^{\tilde \rho}\setminus \overline{D_1}}|{\nabla}^2 w_1|^2
  \leq
  \\
  \leq
    \gamma^{-1}
    \int_{\partial \tilde D_2^{\tilde \rho} }
    \left ( {\rm div}({\mathbb P} \nabla^2 w_1)\cdot \nu+(({\mathbb P} \nabla^2
  w_1)\nu\cdot \tau),_s \right )w_1 - (({\mathbb P} \nabla^2 w_1)\nu\cdot \nu)
  w_{1,\nu}
   =
   \\
   =
   \gamma^{-1}
   \int_{\partial \tilde D_2^{\tilde \rho}}
   \left ( {\rm div}({\mathbb P} \nabla^2 w_1)\cdot \nu+(({\mathbb P} \nabla^2
  w_1)\nu\cdot \tau),_s \right )w_2 - (({\mathbb P} \nabla^2 w_1)\nu\cdot \nu)
  w_{2,\nu}+
  \\
  +
  \gamma^{-1}
   \int_{\partial \tilde D_2^{\tilde \rho}}
   \left ( {\rm div}({\mathbb P} \nabla^2 w_1)\cdot \nu+(({\mathbb P} \nabla^2
  w_1)\nu\cdot \tau),_s \right )w - (({\mathbb P} \nabla^2 w_1)\nu\cdot \nu)
  w,_{\nu}
   \leq
   \\
   \leq
    C \rho_0^2
  \|\widehat{M}\|_{H^{- \frac{1}{2}  }(\partial \Omega, \R^2)}^2
  \left(
  \frac{\tilde \rho}{\rho_0}
    +
  \frac{\rho_0}{\tilde \rho}
  \tilde\epsilon ^{\frac{\tau}{4}\delta^s}
  \right)
  =
    \rho_0^2
 \|\widehat{M}\|_{H^{- \frac{1}{2}  }(\partial \Omega, \R^2)}^2
  \omega(\tilde \epsilon).
\end{multline}

In case Ic), by using \eqref{eq:reg_3}, we have
\begin{equation}
    \label{eq:Ic1}
|D_2\setminus \overline{D_1}|\leq  |\tilde D_1^{\tilde
\rho}\setminus D_1| \leq C\rho_0\tilde\rho,
\end{equation}
with $C$ only depending on $M_0$, $M_1$. By \eqref{eq:Ic1} and by
\eqref{eq:reg_classica}, we have
\begin{equation}
    \label{eq:Ic2}
\int_{D_2\setminus \overline{D_1}} |{\nabla}^2 w_1|^2 \leq
C\rho_0\tilde\rho\|\widehat{M}\|_{H^{- \frac{1}{2}  }(\partial
\Omega, \R^2)}^2,
\end{equation}
with $C$ only depending on $\gamma$, $M$, $M_0$,
$M_1$, so that the thesis follows.

Let us consider now case II). In view of the above arguments,
\begin{equation}
    \label{eq:stab19}
    \|w\|_{C^1( \partial \tilde V_{\tilde \rho}\setminus\partial\Omega)}
    \leq
    C\rho_0^2
    \|\widehat{M}\|_{H^{- \frac{1}{2} }(\partial \Omega, \R^2)} \omega(\tilde\epsilon),
\end{equation}
where $C>0$ only depends on $\gamma$, $M$, $\delta_1$, $M_0$,
$M_1$.

Let $z\in \partial\tilde D_1^{\tilde \rho} \cap \tilde
\Gamma_2^{\tilde
  \rho}$. We have that
\begin{equation}
    \label{eq:stab20}
    |w(z)|+\rho_0|\nabla w(z)|
    \leq
    C\rho_0^2
    \|\widehat{M}\|_{H^{- \frac{1}{2} }(\partial \Omega, \R^2)}
    \omega(\tilde\epsilon),
\end{equation}
where $C>0$ only depends on $\gamma$, $M$, $\delta_1$, $M_0$,
$M_1$.
On the other hand, by using the homogeneous Dirichlet conditions
on the boundaries of the rigid inclusions,
\begin{equation}
    \label{eq:stab21}
    |w_i(z)|+\rho_0|\nabla w_i(z)|
    \leq
    C\rho_0^2
    \|\widehat{M}\|_{H^{- \frac{1}{2} }(\partial \Omega, \R^2)}
    \frac{\tilde\rho}{\rho_0}\leq
    C\rho_0^2
    \|\widehat{M}\|_{H^{- \frac{1}{2} }(\partial \Omega, \R^2)}
    \omega(\tilde\epsilon),
\end{equation}
where $C>0$ only depends on $\gamma$, $M$, $\delta_1$, $M_0$,
$M_1$.
By \eqref{eq:stab20}--\eqref{eq:stab21}, we get
\begin{equation}
    \label{eq:stab22}
    |\overline{g}(z)|+\rho_0|\nabla \overline{g}(z)|
    \leq
    C\rho_0^2
    \|\widehat{M}\|_{H^{- \frac{1}{2} }(\partial \Omega, \R^2)}
   \omega(\tilde\epsilon).
\end{equation}
Let $\overline{g}(x)=ax_1+bx_2+c$. Then, by \eqref{eq:stab22}, we
have
\begin{equation}
    \label{eq:asmall}
    |a|
    \leq
    C\rho_0
    \|\widehat{M}\|_{H^{- \frac{1}{2} }(\partial \Omega, \R^2)}
    \omega(\tilde\epsilon),
\end{equation}
\begin{equation}
    \label{eq:bsmall}
    |b|
    \leq
    C\rho_0
    \|\widehat{M}\|_{H^{- \frac{1}{2} }(\partial \Omega, \R^2)}
    \omega(\tilde\epsilon),
\end{equation}
\begin{equation}
    \label{eq:csmall}
    |c|\leq|\overline{g}(z)|+|a||z|+|b||z|
    \leq
    C\rho_0^2
    \|\widehat{M}\|_{H^{- \frac{1}{2} }(\partial \Omega, \R^2)}
    \omega(\tilde\epsilon),
\end{equation}
where we are assuming for simplicity that the origin belongs to
$\Omega$. It follows that
\begin{equation}
    \label{eq:stab23}
    \|\overline{g}\|_{C^1(\overline{\Omega})}\leq
    C\rho_0^2
    \|\widehat{M}\|_{H^{- \frac{1}{2} }(\partial \Omega, \R^2)}
    \omega(\tilde\epsilon).
\end{equation}
By repeating the arguments of Step 1 for $w=w_1-w_2-\overline{g}$,
we have the additional term $|\overline{g}|$ which is controlled
by \eqref{eq:stab23}, and the thesis follows similarly.

The proof of the better rate of convergence \eqref{eq:omega_cauchy_log}, under
Lipschitz regularity condition on the connected components of $\partial G$
can be obtained by merging the geometrical construction illustrated in \cite{M-R1}
and the arguments seen above.

\end{proof}

\section{Proof of Lemma \ref{lem:H^3} and Lemma \ref{lem:H^4}}
\label{sec: regularity}

Let us denote by
\begin{equation}
    \label{eq:r2.e1}
    B_{\sigma}^+=
    \{(y_1,y_2) \in \R^2 | \ y_1^2+y_2^2 < \sigma^2, \ y_2>0 \}
\end{equation}
the hemidisk of radius $\sigma$, $\sigma > 0$, and let
\begin{equation}
    \label{eq:r2.e2}
    \Gamma_{\sigma}= \{(y_1,y_2) \in \R^2 | \ -\sigma\leq y_1 \leq \sigma, \ y_2=0\},
\end{equation}
\begin{equation}
    \label{eq:r2.e3}
    \Gamma_{\sigma}^+ = \partial B_{\sigma}^+ \setminus
    \Gamma_{\sigma}.
\end{equation}
Moreover, let
\begin{equation}
    \label{eq:r2.e3bis}
    H^2_{ \Gamma_{\sigma}^+}  (B_{\sigma}^+)= \left\{ g \in H^2
    (B_{\sigma}^+) | \  g=0, \ \frac{\partial g}{\partial n}=0 \ on \
    \Gamma_{\sigma}^+ \right\}.
\end{equation}
\begin{proof} [Proof of Lemma \ref{lem:H^3}]
It is not restrictive to assume $\rho_0=1$. By the regularity of
$\partial \Omega$ and $\partial D$, and by \eqref{eq:compactness},
we can construct a finite collection of open sets $\Omega_0$,
$\{\Omega_j\}_{j=1}^N$, $\{\Omega'_i\}_{i=1}^{N'}$ such that
\begin{equation}
    \label{eq:r2.e4a}
    (\Omega_j \cap \Omega) \subset (\Omega \setminus \overline{D}), \ j=1,...,N,
    \quad (\Omega'_i \cap \Omega) \subset (\Omega \setminus \overline{D}), \ i=1,...,N',
\end{equation}
\begin{equation}
    \label{eq:r2.e4b}
    \Omega_j \cap \Omega'_i = \emptyset, \quad i=1,...,N', \
    j=1,...,N,
\end{equation}
\begin{equation}
    \label{eq:r2.e4}
    \Omega \setminus \overline{D} = \Omega_0 \cup
    \left ( \cup_{j=1}^N {\cal{T}}_{(j)}^{-1}(B_{ \frac{1}{2}}^+) \right
    ) \cup
    \left ( \cup_{i=1}^{N'} {\cal{T'}}_{(i)}^{-1}(B_{ \frac{1}{2}}^+) \right
    ),
\end{equation}
\begin{equation}
    \label{eq:r2.e4c}
    \Omega_0 \subset ( \Omega \setminus
    \overline{D})_{\delta_0},
\end{equation}
where $\delta_0$ only depends on $M_0$. Here, ${\cal{T}}_{(j)}$,
$j=1,...,N$, is a homeomorphism of $C^{2,1}$ class which maps
$\Omega_j \cap \Omega$ into $B_1^+$, $\overline{\Omega}_j\cap
\partial \Omega$ into $\Gamma_1$ and $\partial \Omega_j \cap
(\Omega \setminus \overline{D})$ into $\Gamma_1^+$. Similarly,
${\cal{T'}}_{(i)}$ is an homeomorphism of $C^{2,1}$ class which
maps $\Omega'_i \cap \Omega$ into $B_1^+$,
$\overline{\Omega'}_i\cap
\partial D$ into $\Gamma_1$ and $\partial \Omega'_i \cap (\Omega
\setminus \overline{D})$ into $\Gamma_1^+$, $i=1,...,N'$. It can
be shown that every mapping ${\cal{T}}_{(j)}$, ${\cal{T'}}_{(i)}$,
$i=j,...,N$, $i=1,...,N'$, can be chosen such that the Jacobian of
the transformation is identically equal to one, see \cite{l:agmon}
(p. 129). By the regularity of $\partial \Omega$ and $\partial D$
and by \eqref{eq:bound_area}, the numbers $N$, $N'$ are controlled
by a constant only depending on $M_0$ and $M_1$.

By covering $\Omega_0$ with a finite number of spheres contained
in $\Omega \setminus \overline{D}$ and using local interior
regularity results (see, for instance, \cite{M-R-V1}, Theorem
$8.3$), we have that $w\in H^3(\Omega_0)$ and
\begin{equation}
    \label{eq:r3.e1}
    \|w\|_{H^3(\Omega_0)} \leq C \|w\|_{H^2(\Omega \setminus
    \overline{D})},
\end{equation}
where the constant $C>0$ only depends on $M_1$,
$\|\mathbb{P}\|_{C^{0,1}(\overline{\Omega})}$ and $\gamma$.

Let us fix $j$, $1\leq j\leq N$, and let us show that an analogous
estimate holds true for $\Omega_j \cap \Omega$ near the boundary
$\partial \Omega$ where the Neumann data $\widehat{M}$ is given.

The function $w$, solution to
\eqref{eq:dir-pbm-incl-rig-1}-\eqref{eq:equil-rigid-incl},
satisfies
\begin{multline}
    \label{eq:r3.e2}
    \int_{\Omega_j \cap \Omega } \mathbb P \nabla^2 w \cdot \nabla^2 \varphi =
    -\int_{\overline{ \Omega}_j \cap \partial \Omega}
    \left (
    \widehat{M}_n\varphi_{,n} + (\widehat{M}_\tau)_{,s}\varphi
    \right ) ds,  \\ \hbox{for every }
    \varphi \in H^2_{\partial \Omega_j \cap \Omega} (\Omega_j \cap \Omega).
\end{multline}
We define
\begin{equation}
    \label{eq:r4.e1}
    y= {\cal T}_{(j)} (x), \qquad y \in B_{1}^+,
\end{equation}
\begin{equation}
    \label{eq:r4.e2}
    x= {\cal T}_{(j)}^{-1} (y), \qquad x \in \Omega_j \cap \Omega,
\end{equation}
\begin{equation}
    \label{eq:r4.e3}
    u(y) = w ( {\cal T}_{(j)}^{-1} (y)).
\end{equation}
Then, changing the variables in \eqref{eq:r3.e2}, the function $u$
belongs to $H^2(B_{1}^+)$ and satisfies
\begin{equation}
    \label{eq:r4.e4}
    a_+(u, \psi)=l_+(\psi), \qquad \hbox{for every } \psi \in
    H^2_{\Gamma_1^+}(B_1^+),
\end{equation}
where
\begin{multline}
    \label{eq:r4.e5}
    a_+(u, \psi)=
    \\
    =\int_{B_{1}^+}
    \left (
    \mathbb Q \nabla^2u \cdot
    \nabla^2 \psi + \mathbb A \nabla^2u \cdot
    \nabla \psi + \mathbb B \nabla u \cdot
    \nabla^2 \psi + \mathbb D \nabla u \cdot
    \nabla \psi
    \right )
    \left | \det \frac{\partial {\cal T}_{(j)}   }{\partial x} \right
    |^{-1} dy,
\end{multline}
\begin{multline}
    \label{eq:r4.e6}
    l_+(\psi)=
    - \int_{\Gamma_1}
    \left (
    \widehat{{\cal{M}}}_n \frac{\partial {\cal T}_{(j)}   }{\partial x}
    \left ( \frac{\partial {\cal T}_{(j)}   }{\partial x} \right
    )^T
    \nabla \psi \cdot \nu
    \left | \left ( \frac{\partial {\cal T}_{(j)}   }{\partial x} \right
    )^{-T} n \right |  +  \right. \\
    \left. +  \widehat{ {\cal{M}}}_{\tau, \xi}
    \left | \frac{\partial {\cal T}_{(j)}   }{\partial x} \tau \right |
    \psi \right ) \left | \frac{\partial {\cal T}_{(j)}   }{\partial x} \tau \right
    |^{-1} d\xi,
\end{multline}
with
\begin{equation}
    \label{eq:r5.e1}
    S_{kr} = \frac{\partial {\cal T}_{(j)k}}{\partial x_r}, \qquad
    R_{ksr} = \frac{\partial^2  {\cal T}_{(j)k}  }{\partial x_s \partial
    x_r},
\end{equation}
\begin{equation}
    \label{eq:r5.e2}
    Q_{lkmn} = \sum_{i,j,r,s=1}^2
    P_{ijrs}S_{ls}S_{kr}S_{mi}S_{nj},
\end{equation}
\begin{equation}
    \label{eq:r5.e3}
    A_{lkn} = \sum_{i,j,r,s=1}^2
    P_{ijrs}S_{ls}S_{kr}R_{nij},
\end{equation}
\begin{equation}
    \label{eq:r5.e4}
    B_{kmn} = \sum_{i,j,r,s=1}^2
    P_{ijrs}S_{mi}S_{nj}R_{ksr},
\end{equation}
\begin{equation}
    \label{eq:r5.e5}
    D_{kn} = \sum_{i,j,r,s=1}^2
    P_{ijrs}R_{ksr}R_{nij}
\end{equation}
and
\begin{equation}
    \label{eq:r5.e6}
    \psi(y) = \varphi({\cal T}_{(j)}^{-1} (y)), \qquad \psi \in H_{\Gamma_1^+}^2(B_1^+),
\end{equation}
\begin{equation}
    \label{eq:r5.e7}
    \widehat{{\cal{M}}}_n(y)= \widehat{M}_n ( {\cal T}_{(j)}^{-1} (y) ), \qquad
    \widehat{{\cal{M}}}_\tau(y)= \widehat{M}_\tau ( {\cal T}_{(j)}^{-1} (y) )
\end{equation}
and where $\nu$ is the unit outer normal to $B_1^+$. By
\eqref{eq:r5.e2} and \eqref{eq:sym-conditions-C-components},
\eqref{eq:3.convex}, the components of $\mathbb Q$ satisfy the
symmetry conditions
\begin{equation}
    \label{eq:r5.e8}
    Q_{\alpha \beta \gamma \delta} = Q_{ \gamma \delta \alpha \beta} =
    Q_{ \gamma \delta \beta \alpha },  \quad \alpha, \beta, \gamma, \delta
    =1,2, \hbox{ in } \overline{B_1^+},
\end{equation}
and the strong convexity condition
\begin{equation}
    \label{eq:r5.e9}
    \mathbb Q A \cdot A \geq \gamma_0 |A|^2, \hbox{ in }
    \overline{B_1^+},
\end{equation}
for every $2 \times 2$ symmetric matrix $A$, where $\gamma_0$,
$\gamma_0>0$, is a constant only depending on $\gamma$ and $M_0$.

By the regularity assumptions on $\mathbb P$, the tensors $\mathbb
Q $, $\mathbb A$, $\mathbb B$, $\mathbb D$ belong to
$C^{0,1}(\overline{B_1^+})$. Moreover, $\widehat{{\cal{M}}}_n \in
H^{\frac{1}{2}}(\Gamma_1)$ and $\widehat{{\cal{M}}}_{\tau,\xi} \in
H^{-\frac{1}{2}}(\Gamma_1)$.

Now, we use the following regularity result up to the boundary for
the solution $u$ to the problem \eqref{eq:r4.e4}, which is proved
in the Appendix.
\begin{theo}[Boundary regularity for non-homogeneous Neumann conditions]
  \label{theo:BoundNeuNonhom}
Let $u \in H^2(B_1^+)$ defined by \eqref{eq:r4.e3} be the solution
to \eqref{eq:r4.e4}, where the tensors $\mathbb Q $, $\mathbb A$,
$\mathbb B$, $\mathbb D$ and the couple field
$\widehat{{\cal{M}}}$ are defined as above. Then, $u\in H^3\left
(B_{\frac{1}{2}}^+\right )$ and we have
\begin{equation}
    \label{eq:r6bis.e1}
    \|u\|_{H^3\left (B_{\frac{1}{2}}^+\right )} \leq
    C
    \left (
    \|\widehat{\cal{M}}\|_{H^{\frac{1}{2}}(\Gamma_1)} +
    \|u\|_{H^2(B_1^+)}
    \right ),
\end{equation}
where $C>0$ only depends on $M_0$,
$\|\mathbb{P}\|_{C^{0,1}(\overline{\Omega})}$ and $\gamma$.
\end{theo}
By applying the homeomorphism ${\cal T}_{(j)}$ to
\eqref{eq:r6bis.e1} we have
\begin{equation}
    \label{eq:r6.e2}
    \|w\|_{H^3 \left ({\cal{T}}_{(j)}^{-1}\left (B_{\frac{1}{2}}^+\right ) \right )} \leq C \left (
    \| \widehat{M}  \|_{H^{ \frac{1}{2} } ( {\Omega}_j \cap \partial\Omega  )} +
    \|w\|_{ H^2 (\Omega_j \cap \Omega)} \right ),
\end{equation}
where the constant $C>0$ only depends on $M_0$, $\gamma$ and
$\|\mathbb P\|_{C^{0,1}( \overline{\Omega})}$.

We now derive an estimate analogous to \eqref{eq:r6.e2} near the
boundary of the rigid inclusion $D$. Let us fix $i$, $1 \leq i
\leq N'$. The function $w$, solution to
\eqref{eq:dir-pbm-incl-rig-1}-\eqref{eq:equil-rigid-incl},
satisfies
\begin{equation}
    \label{eq:r7.e1}
    \int_{\Omega'_i \cap \Omega } \mathbb P \nabla^2 w \cdot \nabla^2 \varphi =0, \qquad \hbox{for every }
    \varphi \in H^2_0 (\Omega'_i \cap \Omega).
\end{equation}
Then, by introducing the transformation
\begin{equation}
    \label{eq:r7.e2}
    y= {\cal {T'}}_{(i)} (x), \qquad y \in B_{1}^+,
\end{equation}
\begin{equation}
    \label{eq:r7.e3}
    x= {\cal {T'}}_{(i)}^{-1} (y), \qquad x \in \Omega'_i \cap \Omega,
\end{equation}
\begin{equation}
    \label{eq:r7.e4}
    v(y) = w ( {\cal {T'}}_{(i)}^{-1} (y)),
\end{equation}
and changing the variables in \eqref{eq:r7.e1}, the function $v$
belongs to $H^2_{\Gamma_1}(B_{1}^+)$ and satisfies
\begin{equation}
    \label{eq:r7.e5}
    a'_+(v, \psi)=0, \qquad \hbox{for every } \psi \in
    H^2_0(B_1^+),
\end{equation}
where
\begin{multline}
    \label{eq:r7.e6}
    a'_+(v, \psi)=
    \\
    =\int_{B_{1}^+}
    \left (
    \mathbb {Q'} \nabla^2v \cdot
    \nabla^2 \psi + \mathbb {A'} \nabla^2v \cdot
    \nabla \psi + \mathbb {B'} \nabla v \cdot
    \nabla^2 \psi + \mathbb {D'} \nabla v \cdot
    \nabla \psi
    \right )
    \left | \det \frac{\partial {\cal {T'}}_{(i)}   }{\partial x} \right
    |^{-1} dy
\end{multline}
and the tensors $\mathbb {Q'} $, $\mathbb {A'}$, $\mathbb {B'}$,
$\mathbb {D'}$ and the function $\psi$ are defined as in
\eqref{eq:r5.e1}-\eqref{eq:r5.e5} and \eqref{eq:r5.e6},
respectively, with ${\cal {T}}_{(j)}$ replaced by ${\cal
{T'}}_{(i)}$. We note that the tensor $\mathbb {Q'} $ satisfies
the conditions \eqref{eq:r5.e8}, \eqref{eq:r5.e9} and all the
tensors $\mathbb {Q'} $, $\mathbb {A'}$, $\mathbb {B'}$, $\mathbb
{D'}$ belong to $C^{0,1}(\overline{B_1^+})$.

To this point we make use of the following regularity result up to
the boundary $\partial D$ for the solution $v$ to the problem
\eqref{eq:r7.e5}.
\begin{theo}[Boundary regularity for homogeneous Dirichlet conditions]
  \label{theo:BoundDirHom}
Let $v \in H_{\Gamma_1}^2(B_1^+)$ defined by \eqref{eq:r7.e4} be
the solution to \eqref{eq:r7.e5}, where the tensors $\mathbb {Q'}
$, $\mathbb {A'}$, $\mathbb {B'}$, $\mathbb {D'}$ are defined as
above. Then, $v\in H^3\left (B_{\frac{1}{2}}^+\right )$ and we
have
\begin{equation}
    \label{eq:r8bis.e1}
    \|v\|_{H^3\left (B_{\frac{1}{2}}^+\right )} \leq
    C
    \|v\|_{H^2 (B_1^+)},
\end{equation}
where $C>0$ only depends on $M_0$,
$\|\mathbb{P}\|_{C^{0,1}(\overline{\Omega})}$ and $\gamma$.
\end{theo}
The proof of Theorem \ref{theo:BoundDirHom} follows the same
guidelines of the proof of Theorem \ref{theo:BoundNeuNonhom}, see
also \cite{l:agmon} for details.

By applying the homeomorphism ${\cal {T'}}_{(i)}$ to
\eqref{eq:r8bis.e1} we have
\begin{equation}
    \label{eq:r8.e2}
    \|w\|_{H^3 \left({\cal{T'}}_{(i)}^{-1}\left (B_{\frac{1}{2}}^+\right )\right )} \leq C
    \|w\|_{ H^2 (\Omega'_i \cap \Omega)} ,
\end{equation}
where the constant $C>0$ only depends on $M_0$, $\gamma$ and
$\|\mathbb P\|_{C^{0,1}( \overline{\Omega})}$.

Finally, estimate \eqref{eq:reg_H3} follows {}from
\eqref{eq:r2.e4}, \eqref{eq:r3.e1}, \eqref{eq:r6.e2},
\eqref{eq:r8.e2} and {}from the estimate \eqref{eq:w_stima_H2}.
\end{proof}

The proof of Lemma \ref{lem:H^4} follows {}from the following
local version of the $H^4$-regularity near a boundary with
homogeneous Dirichlet data for the solution to the problem
\eqref{eq:r7.e5}.
\begin{theo}
  \label{theo:BoundDirHom-H4}
Let $v \in H_{\Gamma_1}^2(B_1^+)$ defined by \eqref{eq:r7.e4} be
the solution to \eqref{eq:r7.e5}. Then, $v\in
H^4(B_{\frac{1}{2}}^+)$ and we have
\begin{equation}
    \label{eq:r25.e1}
    \|v\|_{H^4\left (B_{\frac{1}{2}}^+\right )} \leq
    C
    \|v\|_{H^2(B_1^+)},
\end{equation}
where $C>0$ only depends on $M_0$,
$\|\mathbb{P}\|_{C^{1,1}(\overline{\Omega})}$ and $\gamma$.
\end{theo}
\begin{proof}
By Theorem \ref{theo:BoundDirHom}, the function $v \in
H^3(B_{\frac{1}{2}}^+)$ satisfies the estimate
\eqref{eq:r8bis.e1}, e.g. $\|v\|_{H^3\left
(B_{\frac{1}{2}}^+\right )} \leq C \|v\|_{H^2(B_1^+)}$, where
$C>0$ only depends on $M_0$,
$\|\mathbb{P}\|_{C^{0,1}(\overline{\Omega})}$ and $\gamma$.

In the first step of the proof we estimate the tangential
derivative $ \frac{\partial \nabla^3 v}{\partial y_1}$. By the
weak formulation of the problem \eqref{eq:r7.e5} it can be shown
that the function $v_{,\alpha} = \frac{\partial v}{\partial
y_{\alpha}} \in H^2\left (B_{\frac{1}{2}}^+\right )$,
$\alpha=1,2$, satisfies the following equation
\begin{equation}
    \label{eq:r25.e2}
    a'_+(v_{,\alpha}, \psi)=l'_{\alpha +}(\psi), \qquad \hbox{for every } \psi \in
    H^2_0\left (B_{\frac{1}{2}}^+\right ),
\end{equation}
where
\begin{multline}
    \label{eq:r26.e1}
    a'_+(v_{,\alpha}, \psi)
    =\int_{B_{\frac{1}{2}}^+}
    \left (
    \mathbb {Q'} \nabla^2 v_{,\alpha} \cdot
    \nabla^2 \psi + \mathbb {A'} \nabla^2 v_{,\alpha} \cdot
    \nabla \psi + \right . \\
    \left . + \mathbb {B'} \nabla v_{,\alpha} \cdot
    \nabla^2 \psi + \mathbb {D'} \nabla v_{,\alpha} \cdot
    \nabla \psi
    \right )
    \left | \det \frac{\partial {\cal {T'}}_{(i)}   }{\partial x} \right
    |^{-1} dy
\end{multline}
\begin{multline}
    \label{eq:r26.e2}
    l'_{\alpha +}(\psi)
    =\int_{B_{\frac{1}{2}}^+}
    \left (
    \divrg(\mathbb {Q}'_{,\alpha} \nabla^2 v) \cdot
    \nabla^2 \psi - \mathbb {A'}_{,\alpha} \nabla^2 v \cdot
    \nabla \psi + \right. \\
    \left. +(B'_{kmn,\alpha}v_{,n})_{,k} \psi_{,m} - \mathbb {D'}_{,\alpha} \nabla v \cdot
    \nabla \psi
    \right )
    \left | \det \frac{\partial {\cal {T'}}_{(i)}   }{\partial x} \right
    |^{-1} dy.
\end{multline}
By the regularity assumptions on $\mathbb P$ and by
\eqref{eq:r8bis.e1} we have
\begin{equation}
    \label{eq:r26.e3}
    |l'_{\alpha +}(\psi)|
    \leq C \| v\|_{H^2 (B_1^+)} \| \psi\|_{L^2\left (B_{\frac{1}{2}}^+\right )},
\end{equation}
$\alpha=1,2$, where the constant $C>0$ only depends on $M_0$,
$\|\mathbb{P}\|_{C^{1,1}(\overline{\Omega})}$ and $\gamma$.

To this point, the arguments used in the proof of
\eqref{eq:r21.e1} for the Neumann case can be adapted to the
Dirichlet boundary condition case to obtain
\begin{equation}
    \label{eq:r27.e1}
    \| \frac{\partial}{\partial y_1}  \nabla^2 v_{,\alpha}\|_{L^2\left (B_{\frac{1}{2}}^+\right )}
    \leq C
    \|v\|_{H^2(B_1^+)},
\end{equation}
$\alpha=1,2$, where the constant $C>0$ only depends on $M_0$,
$\|\mathbb{P}\|_{C^{1,1}(\overline{\Omega})}$ and $\gamma$.

The estimate of the normal derivative $\frac{\partial}{\partial
y_2}  \nabla^2 v_{,\alpha}$, $\alpha=1,2$, can be obtained by
adapting the proof of \eqref{eq:r23.e4}. We obtain
\begin{equation}
    \label{eq:r27.e2}
    \| \frac{\partial}{\partial y_2}  \nabla^2 v_{,\alpha}\|_{L^2\left (B_{\frac{1}{2}}^+\right )}
    \leq C
    \|v\|_{H^2(B_1^+)},
\end{equation}
$\alpha=1,2$, where the contant $C>0$ only depends on $M_0$,
$\|\mathbb{P}\|_{C^{1,1}(\overline{\Omega})}$ and $\gamma$. By
\eqref{eq:r27.e1} and \eqref{eq:r27.e2} we have the thesis.
\end{proof}

\section{Appendix}

In this Appendix we prove Theorem \ref{theo:BoundNeuNonhom}.

\begin{proof} [Proof of Theorem \ref{theo:BoundNeuNonhom}]
As a first step of the proof, let us estimate the partial
derivative $ \frac{\partial}{\partial y_1} \nabla^2 u$ in the
direction parallel to the flat boundary $\Gamma_1$ of $B_1^+$.

Let $\vartheta \in C_0^{\infty}(\R^2)$ be a function such that $0
\leq \vartheta(y) \leq 1$ for every $y \in \R^2$, $\vartheta
\equiv 1$ in $B_{\rho}$, $\vartheta \equiv 0 $ in $\R^2 \setminus
B_{\sigma_0}$ and $|\nabla^k \vartheta| \leq C$, $k=1,...,4$,
where $\rho= \frac{3}{4}$, $\sigma_0=
\frac{1}{2}(\rho+1)=\frac{7}{8}$ and $C$ is an absolute constant.

For every function $\psi \in H_{\Gamma_1^+}^2(B_{1}^+)$, we still
denote by $\psi \in H^2(\R^2_+)$ its extension to $\R^2_+$
obtained by taking $\psi \equiv 0$ in $\R^2_+ \setminus B_{1}^+$.

Let $s$ be a real number different from zero. The difference
operator in the $y_1$-direction is defined by
\begin{equation}
    \label{eq:r10.e1}
    (\tau_{1,s}f)(y) = \frac{f(y+se_1)-f(y)}{s},
\end{equation}
for any function $f$. In what follows we shall assume that $|s|
\leq \frac{1}{16}$. Let us notice that if $u \in H^2(B_1^+)$, then
$\tau_{1,s} (\vartheta u) \in H^2_{\Gamma_1^+}(B_1^+)$.

We evaluate the bilinear form $a_+(\cdot, \psi)$ with $u$ replaced
by $\tau_{1,s}(\vartheta u)$ and $\psi \in
H^2_{\Gamma_1^+}(B_1^+)$. Since
\begin{equation}
    \label{eq:r11.e1}
    \nabla^k (\tau_{1,s} (\vartheta u)) = \tau_{1,s} ( \nabla^k (\vartheta
    u)), \quad \hbox{in } B_1^+, \ k=1,2,
\end{equation}
we have
\begin{multline}
    \label{eq:r11.e2}
    a_+(\tau_{1,s}(\vartheta u), \psi)= \int_{B_{1}^+} \left ( \mathbb Q \tau_{1,s}(\nabla^2(\vartheta
    u))\cdot\nabla^2\psi + \mathbb A \tau_{1,s}(\nabla^2(\vartheta
    u)) \cdot \nabla \psi + \right. \\
    \left. + \mathbb B \tau_{1,s}(\nabla(\vartheta u)) \cdot
    \nabla^2 \psi + \mathbb D \tau_{1,s}(\nabla (\vartheta
    u))\cdot \nabla \psi \right ) dy,
\end{multline}
where we have taken into account that $ \det  \frac{\partial
{\cal{T}}_{(j)}}{\partial x} =1$.

Let us consider the leading term of \eqref{eq:r11.e2}. We
elaborate the corresponding expression by moving the difference
operator {}from $\vartheta u$ to $\vartheta \psi$. We have
\begin{multline}
    \label{eq:r12.e1}
    \int_{B_{1}^+} \mathbb Q \tau_{1,s}(\nabla^2(\vartheta
    u))\cdot\nabla^2\psi = \int_{B_{1}^+} \mathbb Q  \tau_{1,s} ( \vartheta \nabla^2 u)\cdot\nabla^2\psi
    +\\
    + \int_{B_{1}^+} \mathbb Q  \tau_{1,s} (\nabla \vartheta \otimes \nabla u + \nabla u \otimes \nabla
    \vartheta) \cdot \nabla^2 \psi
    + \int_{B_{1}^+} \mathbb Q  \tau_{1,s}(u\nabla^2 \vartheta)\cdot \nabla^2
    \psi \equiv I_1+I_2+I_3.
\end{multline}
The last two terms on the right hand side of \eqref{eq:r12.e1} can
be estimated as follows
\begin{equation}
    \label{eq:r12.e2}
    |I_2| \leq C \|u\|_{H^2(B_{1}^+)}\|\nabla^2 \psi\|_{L^2(B_{1}^+)},
\end{equation}
\begin{equation}
    \label{eq:r12.e3}
    |I_3| \leq C \|u\|_{H^1(B_{1}^+)}\|\nabla^2 \psi\|_{L^2(B_{1}^+)},
\end{equation}
where the constant $C>0$ only depends on $\| \mathbb
P\|_{L^\infty( {B_{1}^+})}$ and $M_0$.

The term $I_1$ can be written as
\begin{multline}
    \label{eq:r12.e4}
    I_1 = \int_{B_{1}^+} \tau_{1,s} ( \mathbb Q   ( \vartheta \nabla^2 u)) \cdot \nabla^2\psi - \\
    - \int_{B_{1}^+} (\tau_{1,s} \mathbb Q)   ( \vartheta \nabla^2 u)(y+se_{1}) \cdot \nabla^2\psi \equiv I_1'
    + I_1'',
\end{multline}
where
\begin{equation}
    \label{eq:r12.e5}
    |I_1''| \leq C \|\nabla^2 u\|_{L^2(B_{1}^+)} \|\nabla^2 \psi\|_{L^2(B_{1}^+)},
\end{equation}
with $C>0$ constant only depending on $\| \mathbb P\|_{C^{0,1}(
\overline{B_{1}^+})}$ and $M_0$.

The remaining term $I_1'$ can be elaborated as follows
\begin{multline}
    \label{eq:r13.e1}
    I_1' = \int_{B_{1}^+} \tau_{1,s} ( \mathbb Q   ( \vartheta \nabla^2 u)) \cdot \nabla^2\psi =
    - \int_{B_{1}^+} \mathbb Q   ( \vartheta \nabla^2 u) \cdot (\tau_{1,-s} \nabla^2\psi) = \\
    = - \int_{B_{1}^+} \mathbb Q   \nabla^2 u \cdot ( \vartheta \tau_{1,-s} \nabla^2\psi) =
    - \int_{B_{1}^+} \mathbb Q   (\nabla^2 u) \cdot \nabla^2( \vartheta \tau_{1,-s} \psi) + \\
    + \int_{B_{1}^+} \mathbb Q \nabla^2 u \cdot \left ( (\nabla^2 \vartheta) \tau_{1,-s} \psi +
    \nabla \vartheta \otimes \nabla(\tau_{1,-s} \psi) + \nabla(\tau_{1,-s} \psi)\otimes \nabla \vartheta
    \right ).
\end{multline}
The last two terms of \eqref{eq:r13.e1} can be estimated as
follows:
\begin{equation}
    \label{eq:r13.e1bis}
    \left | \int_{B_{1}^+} \mathbb Q \nabla^2 u \cdot ((\nabla^2 \vartheta) \tau_{1,-s} \psi )
    \right | \leq C \|\nabla^2 u\|_{L^2(B_{1}^+)} \| \nabla
    \psi\|_{L^2(B_{1}^+)},
\end{equation}
\begin{multline}
    \label{eq:r13.e2}
    \left | \int_{B_{1}^+} \mathbb Q \nabla^2 u \cdot
    \left ( \nabla \vartheta \otimes \nabla ( \tau_{1,-s} \psi) + \nabla(\tau_{1,-s} \psi) \otimes  \nabla \vartheta
     \right  )
    \right | \leq \\
    \leq C \|\nabla^2 u\|_{L^2(B_{1}^+)} \| \nabla^2 \psi\|_{L^2(B_{1}^+)},
\end{multline}
where the constant $C>0$ only depends on $\| \mathbb
P\|_{L^\infty( {B_{1}^+})}$ and $M_0$.

Therefore, by \eqref{eq:r12.e2}, \eqref{eq:r12.e3},
\eqref{eq:r12.e5}, \eqref{eq:r13.e1}, \eqref{eq:r13.e1bis} and
\eqref{eq:r13.e2}, the left hand side of  \eqref{eq:r12.e1} can be
written as
\begin{equation}
    \label{eq:r14.e1}
    \int_{B_{1}^+} \mathbb Q \tau_{1,s}(\nabla^2(\vartheta
    u))\cdot\nabla^2\psi =
    - \int_{B_{1}^+} \mathbb Q   \nabla^2 u \cdot \nabla^2( \vartheta \tau_{1,-s} \psi)
    +r_{\mathbb Q},
\end{equation}
where, by the Poincar\`{e} inequality on
$H_{\Gamma_1^+}^2(B_1^+)$, we have
\begin{equation}
    \label{eq:r14.e2}
    |r_{\mathbb Q}| \leq C \|u\|_{H^2(B_{1}^+)} \| \nabla^2
    \psi\|_{L^2(B_{1}^+)},
\end{equation}
where the constant $C>0$ only depends on $\| \mathbb P\|_{C^{0,1}(
\overline{B_{1}^+})}$ and $M_0$.

We can estimate the remaining terms appearing on the right hand
side of \eqref{eq:r11.e2} similarly. Concerning the term involving
the tensor $\mathbb A$, for instance, we have
\begin{multline}
    \label{eq:r14.e3}
    \int_{B_{1}^+} \mathbb A \tau_{1,s}(\nabla^2(\vartheta
    u))\cdot\nabla\psi = \int_{B_{1}^+} \mathbb A  \tau_{1,s} ( \vartheta \nabla^2 u)\cdot\nabla\psi
    +\\
    + \int_{B_{1}^+} \mathbb A  \tau_{1,s} (\nabla \vartheta \otimes \nabla u + \nabla u \otimes \nabla
    \vartheta) \cdot \nabla \psi
    + \int_{B_{1}^+} \mathbb A  \tau_{1,s}(u\nabla^2 \vartheta)\cdot \nabla
    \psi \equiv J_1+J_2+J_3,
\end{multline}
where, by the Poincar\`{e} inequality on $H_{\Gamma_1^+}^2(B_1^+)$
we have
\begin{equation}
    \label{eq:r15.e1}
    |J_2+J_3| \leq C \|u\|_{H^2(B_{1}^+)}\|\nabla \psi\|_{L^2(B_{1}^+)},
\end{equation}
where the constant $C>0$ only depends on $\| \mathbb
P\|_{L^\infty( {B_{1}^+})}$ and $M_0$. The term $J_1$ can be
elaborated as it was done for $I_1$, obtaining
\begin{multline}
    \label{eq:r15.e2}
    J_1 = \int_{B_{1}^+} \tau_{1,s} ( \mathbb A   ( \vartheta \nabla^2 u)) \cdot \nabla\psi - \\
    - \int_{B_{1}^+} (\tau_{1,s} \mathbb A)   ( \vartheta \nabla^2 u)(y+se_{1}) \cdot \nabla\psi \equiv J_1'
    + J_1'',
\end{multline}
where
\begin{equation}
    \label{eq:r15.e3}
    |J_1''| \leq C \|\nabla^2 u\|_{L^2(B_{1}^+)} \|\nabla \psi\|_{L^2(B_{1}^+)},
\end{equation}
with $C>0$ constant only depending on $\| \mathbb P\|_{C^{0,1}(
\overline{B_{1}^+})}$ and $M_0$. Concerning the term $J_1'$ we
have
\begin{multline}
    \label{eq:r16.e1}
    J_1' = \int_{B_{1}^+} \tau_{1,s} ( \mathbb A   ( \vartheta \nabla^2 u)) \cdot \nabla\psi
    = - \int_{B_{1}^+} \mathbb A   \nabla^2 u \cdot ( \vartheta \tau_{1,-s} \nabla\psi)
    =\\
    = - \int_{B_{1}^+} \mathbb A   (\nabla^2 u) \cdot \nabla( \vartheta \tau_{1,-s} \psi)
    + \int_{B_{1}^+} \mathbb A \nabla^2 u \cdot \nabla \vartheta (\tau_{1,-s}
    \psi),
\end{multline}
where the last term of \eqref{eq:r16.e1} can be estimated as
\begin{equation}
    \label{eq:r16.e2}
    \left | \int_{B_{1}^+} \mathbb A \nabla^2 u \cdot \nabla \vartheta (\tau_{1,-s}
    \psi)   \right |
    \leq C
     \|\nabla^2 u\|_{L^2(B_{1}^+)}\|\nabla \psi\|_{L^2(B_{1}^+)},
\end{equation}
with  $C>0$ constant only depending on $\| \mathbb P\|_{L^\infty(
{B_{1}^+})}$ and $M_0$.

Therefore, by \eqref{eq:r15.e1}, \eqref{eq:r15.e3},
\eqref{eq:r16.e1} and \eqref{eq:r16.e2}, the left hand side of
\eqref{eq:r14.e3} can be written as
\begin{equation}
    \label{eq:r16.e3}
    \int_{B_{1}^+} \mathbb A \tau_{1,s}(\nabla^2(\vartheta
    u))\cdot\nabla\psi = - \int_{B_{1}^+} \mathbb A   (\nabla^2 u) \cdot \nabla( \vartheta \tau_{1,-s}
    \psi)+r_{\mathbb A},
\end{equation}
where
\begin{equation}
    \label{eq:r17.e1}
    |r_{\mathbb A}| \leq C \|u\|_{H^2(B_{1}^+)} \| \nabla
    \psi\|_{L^2(B_{1}^+)},
\end{equation}
with a constant $C>0$ only depending on $\| \mathbb P\|_{C^{0,1}(
\overline{B_{1}^+})}$ and $M_0$.

By similar procedure we obtain
\begin{equation}
    \label{eq:r17.e2}
    \int_{B_{1}^+} \mathbb B \tau_{1,s}(\nabla(\vartheta
    u))\cdot\nabla^2\psi = - \int_{B_{1}^+} \mathbb B   \nabla u \cdot \nabla^2( \vartheta \tau_{1,-s}
    \psi)+r_{\mathbb B},
\end{equation}
\begin{equation}
    \label{eq:r17.e3}
    \int_{B_{1}^+} \mathbb D \tau_{1,s}(\nabla(\vartheta
    u))\cdot\nabla\psi = - \int_{B_{1}^+} \mathbb D   \nabla u \cdot \nabla( \vartheta \tau_{1,-s}
    \psi)+r_{\mathbb D},
\end{equation}
where
\begin{equation}
    \label{eq:r17.e4}
    |r_{\mathbb B}| \leq C \|u\|_{H^2(B_{1}^+)} \| \nabla^2
    \psi\|_{L^2(B_{1}^+)},
\end{equation}
\begin{equation}
    \label{eq:r17.e5}
    |r_{\mathbb D}| \leq C \|u\|_{H^2(B_{1}^+)} \| \nabla
    \psi\|_{L^2(B_{1}^+)},
\end{equation}
where the constant $C>0$ only depends on $\| \mathbb P\|_{C^{0,1}(
\overline{B_{1}^+})}$ and $M_0$. Collecting the above results and
by using the Poincar\`{e} inequality on $H^2_{\Gamma_1^+}(B_1^+)$,
by \eqref{eq:r11.e2}, \eqref{eq:r14.e1}, \eqref{eq:r14.e2},
\eqref{eq:r16.e3}-\eqref{eq:r17.e5}, we have
\begin{equation}
    \label{eq:r18.e1}
    a_+(\tau_{1,s}(\vartheta u), \psi)= - a_+( u, \vartheta
    \tau_{1,-s}\psi)+ r_+,
\end{equation}
where
\begin{equation}
    \label{eq:r18.e2}
    |r_+| \leq C \|u\|_{H^2(B_{1}^+)} \| \nabla^2
    \psi\|_{L^2(B_{1}^+)},
\end{equation}
with $C>0$ constant only depending on $\| \mathbb P\|_{C^{0,1}(
\overline{B_{1}^+})}$ and $M_0$. Since $\psi \in
H^2_{\Gamma_1^+}(B_1^+)$, the function $\vartheta \tau_{1,-s}\psi
\in H^2_{\Gamma_1^+}(B_1^+)$ is a test function and then, by the
weak formulation of the problem \eqref{eq:r4.e4}, we have
\begin{multline}
    \label{eq:r18.e3}
    a_+(u, \vartheta \tau_{1,-s}\psi)=l_+(\vartheta
    \tau_{1,-s}\psi)= \\
    =-\int_{\Gamma_1}
    \left (
    \widehat{{\cal{M}}}_n \frac{\partial {\cal T}_{(j)}   }{\partial x}
    \left ( \frac{\partial {\cal T}_{(j)}   }{\partial x} \right
    )^T
    \nabla (\vartheta \tau_{1,-s}\psi) \cdot \nu
    \left | \left ( \frac{\partial {\cal T}_{(j)}   }{\partial x} \right
    )^{-T} n \right |  +  \right. \\
    \left. +  \widehat{ {\cal{M}}}_{\tau, \xi}
    \left | \frac{\partial {\cal T}_{(j)}   }{\partial x} \tau \right |
    (\vartheta \tau_{1,-s}\psi) \right ) \left | \frac{\partial {\cal T}_{(j)}   }{\partial x} \tau \right
    |^{-1} d\xi.
\end{multline}
By trace inequalities and Poincar\`{e} inequality we have
\begin{multline}
    \label{eq:r19.e1}
    |l_+(\vartheta \tau_{1,-s}\psi)| \leq \\
    \leq C \left (
    \| \widehat{{\cal{M}}}_n \|_{H^{\frac{1}{2}}(\Gamma_1)} \|\nabla(\vartheta
    \tau_{1,-s}\psi)\|_{H^{-\frac{1}{2}}(\Gamma_1)} + \| \widehat{ {\cal{M}}}_{\tau, \xi} \|_{H^{-\frac{1}{2}}(\Gamma_1)} \|\vartheta
    \tau_{1,-s}\psi\|_{H^{\frac{1}{2}}(\Gamma_1)} \right ) \leq \\
    \leq C \left (
    \| \widehat{{\cal{M}}}_n \|_{H^{\frac{1}{2}}(\Gamma_1)} \|\nabla(\vartheta
    \tau_{1,-s}\psi)\|_{L^2(B_1^+)} + \| \widehat{ {\cal{M}}}_{\tau} \|_{H^{\frac{1}{2}}(\Gamma_1)} \|\vartheta
    \tau_{1,-s}\psi\|_{H^{1}(B_1^+)} \right ) \leq \\
    \leq C  \| \widehat{{\cal{M}}} \|_{H^{\frac{1}{2}}(\Gamma_1)} \|\nabla^2\psi\|_{L^{2}(B_1^+)}
    ,
\end{multline}
where the constant $C>0$ only depends on $M_0$. By
\eqref{eq:r18.e1}--\eqref{eq:r19.e1} we have
\begin{equation}
    \label{eq:r19.e2}
    a_+(\tau_{1,s}(\vartheta u), \psi) \leq C_1
    \left (
    \| \widehat{{\cal{M}}} \|_{H^{\frac{1}{2}}(\Gamma_1)} +
    \|u\|_{H^2(B_{1}^+)}
    \right )
    \| \nabla^2
    \psi\|_{L^2(B_{1}^+)},
\end{equation}
for every $\psi \in H^2_{\Gamma_1^+}(B_1^+)$, where the constant
$C_1>0$ only depends on $\| \mathbb P\|_{C^{0,1}(
\overline{B_{1}^+})}$ and $M_0$.

Let $\psi \in H^2_{\Gamma_1^+}(B_1^+)$ and let us estimate {}from
below $a_+(\psi,\psi)$. For every $\epsilon>0$ and for every $\psi
\in H^2_{\Gamma_1^+}(B_1^+)$ we have
\begin{multline}
    \label{eq:r19.e3}
    \left |
    \int_{B_1^+} \mathbb A \nabla^2\psi\cdot\nabla\psi +  \mathbb B
    \nabla\psi\cdot\nabla^2\psi  +  \mathbb D \nabla\psi\cdot\nabla\psi \right | \leq \\
    \leq C
    \left (
    \epsilon \|\nabla^2 \psi\|_{L^2(B_1^+)}^2 + \left ( 1 +
    \frac{1}{\epsilon}\right ) \|\nabla \psi\|_{L^2(B_1^+)}^2
    \right ),
\end{multline}
where the constant $C>0$ only depends on $\| \mathbb P\|_{C^{0,1}(
\overline{B_{1}^+})}$ and $M_0$. Therefore, by the strong
convexity of $\mathbb Q$ and choosing $\epsilon$ small enough in
\eqref{eq:r19.e3} we have
\begin{equation}
    \label{eq:r20.e1}
    a_+(\psi,\psi) \geq C_2 \|\nabla^2 \psi\|_{L^2(B_1^+)}^2 - C_3\|\nabla
    \psi\|_{L^2(B_1^+)}^2,
\end{equation}
where $C_2>0$, $C_3>0$ are constants only depending on $\| \mathbb
P\|_{C^{0,1}( \overline{B_{1}^+})}$, $\gamma$ and $M_0$. Now, by
taking $\psi=\tau_{1,s}(\vartheta u)$ in \eqref{eq:r19.e2} and
\eqref{eq:r20.e1}, we obtain
\begin{multline}
    \label{eq:r20.e2}
    C_2\|\nabla^2 (\tau_{1,s}(\vartheta u))\|_{L^2(B_1^+)}^2
    \leq C_3 \|\nabla
    (\tau_{1,s}(\vartheta u))\|_{L^2(B_1^+)}^2 + \\
    + C_1
    \left (
    \| \widehat{{\cal{M}}} \|_{H^{\frac{1}{2}}(\Gamma_1)} +
    \|u\|_{H^2(B_{1}^+)}
    \right )
    \| \nabla^2
    (\tau_{1,s}(\vartheta u))\|_{L^2(B_{1}^+)},
\end{multline}
where the constants $C_i>0$, $i=1,2,3$, only depend on $\| \mathbb
P\|_{C^{0,1}( \overline{B_{1}^+})}$, $\gamma$ and $M_0$.
Therefore, recalling that $\| \nabla (\tau_{1,s}(\vartheta
u))\|_{L^2(B_1^+)}\leq c \|u\|_{H^2(B_1^+)}$, where $c$ is an
absolute constant, and by Poincar\`{e} inequality we have
\begin{equation}
    \label{eq:r20.e2bis}
    \left \|
    \nabla^2 ( \tau_{1,s} (\vartheta u))
    \right \|_{L^2 ( B_{\rho}^+)}
    \leq C
    \left (\| \widehat{{\cal{M}}} \|_{H^{\frac{1}{2}}(\Gamma_1)}  + \| u \|_{H^2 ( B_{1}^+)} \right ),
\end{equation}
where the constant $C>0$ only depends on $\| \mathbb P\|_{C^{0,1}(
\overline{B_{1}^+})}$, $\gamma$ and $M_0$. Taking the limit as $s
\rightarrow 0$ we finally have
\begin{equation}
    \label{eq:r21.e1}
    \left\| \frac{\partial }{\partial y_1}  \nabla^2 u \right\|_{L^2 ( B_{\rho}^+)}
    \leq
    C \left (\| \widehat{{\cal{M}}} \|_{H^{\frac{1}{2}}(\Gamma_1)}  + \| u \|_{H^2 ( B_{1}^+)} \right ),
\end{equation}
where the constant $C>0$ only depends on $\| \mathbb P\|_{C^{0,1}(
\overline{B_{1}^+})}$, $\gamma$ and $M_0$.

To obtain an analogous estimate for the normal derivative $
\frac{\partial}{\partial y_2} \nabla^2 u$, we shall use the
following Lemma.
\begin{lem} (\cite[Lemma $9.3$]{l:agmon})
    \label{lem:agmon}
    Assume that the function $g \in L^2 (B_{\sigma}^+)$ has weak
    tangential derivative $ \frac{\partial g}{\partial y_1} \in L^2
    (B_{\sigma}^+)$ and that there exist a constant $K$, $K>0$, such
    that
\begin{equation}
    \label{eq:r21.e2}
    \left |\int_{B_{\sigma}^+} g \frac{\partial^2 \psi}{\partial
    y_2^2} \right |
    \leq K \| \psi \|_{H^1 ( B_{\sigma}^+)}, \qquad \hbox{for
    every } \psi \in C_0^{\infty}(B_{\sigma}^+).
\end{equation}
Then, for every $\rho < \sigma$, $g \in H^1(B_{\rho}^+)$ and
\begin{equation}
    \label{eq:r21.e3}
    \| g \|_{H^1 ( B_{\rho}^+)}
    \leq
    C
    \left (
    K + \| g \|_{L^2 ( B_{\sigma}^+)} + \sigma\left\|\frac{\partial g}{\partial y_1} \right\|_{L^2 (
    B_{\sigma}^+)}
    \right ),
\end{equation}
where $C>0$ only depends on $\frac{\rho}{\sigma}$.
\end{lem}

Throughout this part let $\psi \in C_0^{\infty}(B_{\sigma_0}^+)$.
The arguments above show that, without loss of generality, we may
assume that $\mathbb A = \mathbb B = \mathbb D =0$. Therefore, we
can write
\begin{equation}
    \label{eq:r22.e1}
    a_{+}(u, \psi) =
    \int_{B_{\sigma_0}^+}
    \left (
    Q_{22\gamma\delta} u,_{\gamma\delta} \psi,_{22}
    +
    2Q_{12\gamma\delta}u,_{\gamma\delta} \psi,_{12} +
    Q_{11\gamma\delta}u,_{\gamma\delta} \psi,_{11}
    \right )
\end{equation}
and, by integration by parts, we have
\begin{multline}
    \label{eq:r22.e2}
    \int_{B_{\sigma_0}^+}
    Q_{22\gamma\delta} u,_{\gamma\delta} \psi,_{22}
    =
    a_{+}(u, \psi)
    +
    2\int_{B_{\sigma_0}^+}
    \left (
    Q_{12 \gamma \delta},_{1}u,_{\gamma\delta}
    \psi,_{2}+Q_{12 \gamma
    \delta}u,_{\gamma\delta 1}\psi,_{2}
    \right ) + \\
    + \int_{B_{\sigma_0}^+} \left (
    Q_{11 \gamma \delta},_{1}u,_{\gamma\delta}
    \psi,_{1}+Q_{11 \gamma
    \delta}u,_{\gamma\delta 1}\psi,_{1}
    \right ).
\end{multline}
Since $\psi \in C_0^\infty ({B_{\sigma_0}^+})$, we have $a_{+}(u,
\psi)=0$. Let us define
\begin{equation}
    \label{eq:322.e3}
    g = \sum_{\gamma, \delta=1}^2 Q_{22\gamma\delta}u,_{\gamma\delta}.
\end{equation}
Clearly $g \in L^2(B_\sigma^+)$ and $\frac{\partial g}{\partial
y_1} \in L^2(B_\sigma^+)$ by \eqref{eq:r21.e1}. By
\eqref{eq:r22.e2} and by estimate \eqref{eq:r21.e1} we have
\begin{equation}
    \label{eq:r23.e1}
    \left|\int_{B_{\sigma_0}^+} g \psi,_{22}\right|
    \leq
    C \left (\| \widehat{{\cal{M}}} \|_{H^{\frac{1}{2}}(\Gamma_1)} + \| u \|_{H^2 ( B_{1}^+)} \right ) \| \psi \|_{H^1 (
    B_{\sigma_0}^+)},
\end{equation}
with $C>0$ constant only depending on $\| \mathbb P\|_{C^{0,1}(
\overline{B_{1}^+})}$, $\gamma$ and $M_0$, so that, in
$B_{\sigma_0}^+$, the function $g$ satisfies the conditions of
Lemma \ref{lem:agmon}. Hence, $g \in H^1(B_{\rho}^+)$ and we have
\begin{equation}
    \label{eq:r23.e2}
    \| g \|_{H^1 ( B_{\rho}^+)} \leq C \left (\| \widehat{{\cal{M}}} \|_{H^{\frac{1}{2}}(\Gamma_1)}
     + \| u \|_{H^2 ( B_{1}^+)} \right ),
\end{equation}
where the constant $C>0$ only depends on $\| \mathbb P\|_{C^{0,1}(
\overline{B_{1}^+})}$, $\gamma$ and $M_0$. By the ellipticity of
$\mathbb Q$, by \eqref{eq:r21.e1} and \eqref{eq:r23.e2}, it
follows that
\begin{equation}
    \label{eq:r23.e3}
    u,_{22} = Q_{2222}^{-1}
    \left
    (
    g -
    \sum_{(\gamma, \delta)\neq (2,2)}
    Q_{22\gamma\delta}u,_{\gamma\delta}
    \right ) \in H^1(B_{\rho}^+),
\end{equation}
and therefore
\begin{equation}
    \label{eq:r23.e4}
    \|u,_{222}\|_{L^2 ( B_{\rho}^+)}
    \leq
    C \left (\| \widehat{{\cal{M}}} \|_{H^{\frac{1}{2}}(\Gamma_1)}
     + \| u \|_{H^2 ( B_{\sigma}^+)} \right ),
\end{equation}
where the constant $C>0$ only depends on $\| \mathbb P\|_{C^{0,1}(
\overline{B_{1}^+})}$, $\gamma$ and $M_0$.

By \eqref{eq:r21.e1} and \eqref{eq:r23.e4} we have
\begin{equation}
    \label{eq:r24.e1}
    \|u\|_{H^3 ( B_{\rho}^+)}
    \leq C \left (\| \widehat{{\cal{M}}} \|_{H^{\frac{1}{2}}(\Gamma_1)}
     + \| u \|_{H^2 ( B_{1}^+)} \right ),
\end{equation}
where the constant $C>0$ only depends on $\| \mathbb P\|_{C^{0,1}(
\overline{B_{1}^+})}$, $\gamma$ and $M_0$. This completes the
proof.
\end{proof}


\begin{thebibliography}{} \label{bbiibb}

\bibitem[Ad]{l:ad}
R. A. Adams, Sobolev Spaces,
Academic Press, New York, 1975.

\bibitem[Ag]{l:agmon} S. Agmon, Lectures on Elliptic Boundary Value Problems, Van Nostrand, New York, 1965.

\bibitem[Al-B-Ro-Ve]{A-B-R-V} G. Alessandrini, E. Beretta, E.
Rosset, S. Vessella, Optimal stability for inverse elliptic
boundary value problems with unknown boundaries, Ann. Scuola Norm.
Sup. Pisa Cl. Sci. 29(4) (2000), 755-806.

\bibitem[Al-R-Ro-Ve]{A-R-R-V} G. Alessandrini, L. Rondi, E. Rosset,
S. Vessella, The stability for the Cauchy problem for elliptic
equations, Inverse Problems 25 (2009), 1-47.

\bibitem[Al-Ro]{Al-Ro} G. Alessandrini, E. Rosset, The inverse conductivity problem with one measurement: bound on the size of the unknown object,
SIAM J. Appl. Math. 58(4) (1998), 1060-1071.

\bibitem[Ali]{Ali} S. Alinhac, Non-unicit\'{e} pour des op\'{e}rateurs diff\'{e}rentiels \`{a}
la caract\'{e}ristiques complexes simples. Ann. Sci. \'{E}cole
Norm. Sup. (4), 13(3) (1980), 385-393.

\bibitem[DiC-R]{DiC-R} M. Di Cristo, L. Rondi, Examples of exponential instability for inverse inclusion and scattering problems, Inverse Problems 19(3) (2003), 685-701.

\bibitem[F-Q]{F-Q-11} W. Fan, P. Qiao, A strain energy-based damage severity correction
factor method for damage identification in plate-type structures,
Preprint, 2011.

\bibitem[Fi]{Fi} G. Fichera, Existence Theorems in Elasticity,
Handbuch der Physik, vol. VI, Springer-Verlag, Berlin, Heidelberg,
New York, 1972.

\bibitem[G-T]{G-T} D. Gilbarg and N.S. Trudinger, Elliptic partial differential
equations of second order, Springer, New York, 1983.

\bibitem[Gu]{Gu} M. E. Gurtin, The Linear Theory of Elasticity,
Handbuch der Physik, vol. VI, Springer-Verlag, Berlin, Heidelberg,
New York, 1972.

\bibitem[M-Ro1]{M-R1} A. Morassi, E. Rosset,
Uniqueness and stability in determining a rigid inclusion
in an elastic body, Mem. Amer. Math. Soc. 200 (938), (2009).

\bibitem[M-Ro2]{M-R2} A. Morassi, E. Rosset,
Unique determination of unknown boundaries in an elastic plate by
one measurement, C. R. Mecanique, doi:10.1016/j.crme.2010.07.011,
2010.

\bibitem[M-Ro-Ve1]{M-R-V1}  A. Morassi, E. Rosset, S. Vessella,
Size estimates for inclusions in an elastic plate by boundary
measurements, Indiana Univ. Math. J. 56(5) (2007), 2535-2384.

\bibitem[M-Ro-Ve2]{M-R-V2}    A. Morassi, E. Rosset, S. Vessella,
Unique determination of a cavity in a elastic plate by two
boundary measurements, Inverse Problems and Imaging 1(3) (2007),
481-506.

\bibitem[M-Ro-Ve3]{M-R-V3}  A. Morassi, E. Rosset, S. Vessella,
Detecting general inclusions in elastic plates, Inverse Problems,
25 (2009), 045009.

\bibitem[M-Ro-Ve4]{M-R-V4} A. Morassi, E. Rosset, S. Vessella,
Sharp three sphere inequality for perturbations of a product of
two second order elliptic operators and stability for the Cauchy
problem for the anisotropic plate equation, J. Funct. Anal. 261(6)
(2011),  1494-1541.

\bibitem[M-Ro-Ve5]{M-R-V5} A. Morassi, E. Rosset, S. Vessella,
Estimating area of inclusions in anisotropic plates from boundary
data, Preprint (2011).



\end{thebibliography}
\end{document}